\documentclass{amsart}

\usepackage{amsmath,amsthm,amstext,amssymb,bbm}
\usepackage{hyperref}
\usepackage{enumerate,enumitem}

\usepackage{todonotes}

\theoremstyle{plain}
\newtheorem{theorem}{Theorem}[section]
\newtheorem{lemma}[theorem]{Lemma}

\newtheorem{corollary}[theorem]{Corollary}
\newtheorem{proposition}[theorem]{Proposition}
\newtheorem{claim}{Claim}
\newtheorem*{claim*}{Claim}
\theoremstyle{definition}
\newtheorem{definition}[theorem]{Definition}
\newtheorem{remark}[theorem]{Remark}

\newtheorem{question}[theorem]{Question}

\newcommand{\betrag}[1]{\vert{#1}\vert}

\newcommand{\crit}[1]{{{\rm{crit}}\left({#1}\right)}}
\newcommand{\cof}[1]{{{\rm{cof}}(#1)}}
\newcommand{\otp}[1]{{{\rm{otp}}\left(#1\right)}}
\newcommand{\ran}[1]{{{\rm{ran}}(#1)}}
\newcommand{\clo}[1]{{{\rm{Cl}}_{#1}}}

\newcommand{\POT}[1]{{\mathcal{P}}({#1})}
\newcommand{\POTI}[2]{{\mathcal{P}}_{{#2}}({#1})}
\newcommand{\map}[3]{{#1}:{#2}\longrightarrow{#3}}

\newcommand{\Set}[2]{\{{#1}~\vert~{#2}\}}
\newcommand{\seq}[2]{\langle{#1}~\vert~{#2}\rangle}

\newcommand{\goedel}[2]{{\prec}{#1},{#2}{\succ}}

\newcommand{\anf}[1]{{\text{``}\hspace{0.3ex}{#1}\hspace{0.3ex}\text{''}}}

\newcommand{\Poti}[2]{{\mathcal{P}}_{#2}(#1)}
\newcommand{\HH}[1]{{\rm{H}}(#1)}
\newcommand{\Ult}[2]{{\mathrm{Ult}}({#1},{#2})}

\newcommand{\id}{{\rm{id}}}
\newcommand{\Lim}{{\rm{Lim}}}
\newcommand{\On}{{\rm{On}}}
\newcommand{\LL}{{\rm{L}}}

\newcommand{\ZFC}{{\rm{ZFC}}}

\newcommand{\PFA}{{\rm{PFA}}}
\newcommand{\ISP}{{\rm{ISP}}}
\newcommand{\SSP}{{\rm{SSP}}}

\newcommand{\PPP}{{\mathbb{P}}}

\newcommand{\VV}{{\rm{V}}}

\newcommand{\calL}{\mathcal{L}}

\newcommand{\calU}{\mathcal{U}}

\title{Small Embedding Characterizations for Large Cardinals}

\author{Peter Holy}
\address{Mathematisches Institut\\Rheinische Friedrich-Wilhelms-Universit\"at Bonn\\En\-de\-nicher Allee 60\\53115 Bonn\\Germany}
\email{pholy@math.uni-bonn.de}

\author{Philipp L\"ucke}
\address{Mathematisches Institut\\Rheinische Friedrich-Wilhelms-Universit\"at Bonn\\En\-de\-nicher Allee 60\\53115 Bonn\\Germany}
\email{pluecke@math.uni-bonn.de}

\author{Ana Njegomir}
\address{Mathematisches Institut\\Rheinische Friedrich-Wilhelms-Universit\"at Bonn\\En\-de\-nicher Allee 60\\53115 Bonn\\Germany}
\email{njegomir@math.uni-bonn.de}

\thanks{During the preparation of this paper, the first two authors were partially supported by DFG-grant LU2020/1-1.}
\subjclass[2010]{03E55, 03E05, 03E35} 
\keywords{Large Cardinals, elementary embeddings, generalized tree properties}

\begin{document}

\begin{abstract}
  We show that many large cardinal notions can be characterized in terms of the existence of certain elementary embeddings between transitive set-sized structures, that map their critical point to the large cardinal in question. In particular, we provide such embedding characterizations also for several large cardinal notions for which no embedding characterizations have been known so far, namely for subtle, for ineffable, and for $\lambda$-ineffable cardinals. As an application, which we will study in detail in a subsequent paper, we present the basic idea of our concept of internal large cardinals. We provide the definition of certain kinds of internally subtle, internally $\lambda$-ineffable and internally supercompact cardinals, and show that these correspond to generalized tree properties, that were investigated by Wei\ss\ in his \cite{weissthesis} and \cite{MR2959668}, and by  Viale and  Wei\ss\ in \cite{MR2838054}. In particular, this yields new proofs of Wei\ss 's results from \cite{weissthesis} and \cite{MR2959668}, eliminating problems contained in the original proofs.
\end{abstract}

\maketitle



\section{Introduction}

Many large cardinal notions are characterized by the existence of non-trivial elementary embeddings with certain properties. There are two kinds of such characterizations, the first, more common one, where the large cardinal property of $\kappa$ is characterized by the existence of elementary embeddings with critical point $\kappa$, and the second, less common one, where the large cardinal property of $\kappa$ is characterized by the existence of elementary embeddings which map their critical point to $\kappa$. We denote characterizations of the latter kind as \emph{small embedding characterizations}. The following classical result of Menachem Magidor is the first example of a characterization of the second kind. Throughout this paper, we call an elementary embedding $\map{j}{M}{N}$ between transitive classes \emph{non-trivial} if there is an ordinal $\alpha\in M$ with $j(\alpha)>\alpha$. In this case, we let $\crit{j}$ denote the least such ordinal.

\begin{theorem}[{\cite[Theorem 1]{MR0295904}}]\label{theorem:MagidorChar}
  A cardinal $\kappa$ is supercompact if and only if for every $\eta>\kappa$, there is a non-trivial elementary embedding $\map{j}{\VV_\alpha}{\VV_\eta}$ with $\alpha<\kappa$ and $j(\crit{j})=\kappa$.
\end{theorem}

Other examples of large cardinal properties that are characterized by the existence of small embeddings are \emph{subcompactness} (introduced by Ronald Jensen) and its generalizations (see \cite{MR3096624}), and also Ralf Schindler's \emph{remarkable cardinals} (see \cite{MR1765054}). In this paper, we will study large cardinal properties that can be characterized by small embeddings of the following form.

\begin{definition}\label{definition:smallembedding}
 Given cardinals $\kappa<\theta$, we say that a non-trivial elementary embedding $\map{j}{M}{\HH{\theta}}$ is a \emph{small embedding for $\kappa$} if $M\in\HH{\theta}$ is transitive, and $j(\crit{j})=\kappa$ holds.
\end{definition}

The properties of cardinals $\kappa$ studied in this paper usually state that for sufficiently large\footnote{Here, $\theta$ being a \emph{sufficiently large cardinal} means that there is an $\alpha\geq\kappa$ such that the corresponding statement holds for all cardinals $\theta>\alpha$.} cardinals $\theta$, there is a small embedding $\map{j}{M}{\HH{\theta}}$ for $\kappa$ with certain elements of $\HH{\theta}$ in its range, and with the property that the domain model $M$ satisfies certain correctness properties with respect to the universe of sets $\VV$, \footnote{We make this requirement mostly to avoid trivial small embedding characterizations. For example, without this requirement, one could propose the following equivalence: $\kappa$ is measurable if and only if there is a transitive $M$ and $\map{j}{M}{\HH{(2^\kappa)^+}}$ such that $j(\crit j)=\kappa$ and $\crit j$ is measurable in $M$. However $\crit j$ will in general not be measurable in $\VV$ (consider for example the least measurable cardinal $\kappa$), hence this trivial characterization is ruled out by the above requirement. We will later present a non-trivial small embedding characterization of measurability (see Lemma \ref{lemma:MeasureableSmallChar}).} sometimes in combination with some kind of smallness assumption about $M$.

The results of this paper will show that many classical large cardinal properties can be characterized in this way. For example, the proof of Theorem \ref{theorem:MagidorChar} directly yields the following small embedding characterization of supercompactness. Note that the requirement that $M=\HH{\delta}$ below can easily be interpreted as a correctness property of $M$ (since $\VV=\HH{\On}$), and that $\delta<\kappa$ is a smallness assumption on $M$.

\begin{corollary}\label{corollary:SupercompactMagidorStyle}
  The following statements are equivalent for every cardinal $\kappa$: 
 \begin{enumerate}[leftmargin=0.7cm]
  \item $\kappa$ is supercompact. 
  
  \item For all sufficiently large cardinals $\theta$, there is a small embedding $\map{j}{M}{\HH{\theta}}$ for $\kappa$ with the property that $M=\HH{\delta}$ for some cardinal $\delta<\kappa$. \qed
 \end{enumerate}
\end{corollary}

In addition, our results will show that the collections of small embeddings witnessing certain large cardinal properties relate in a way that parallels the implication structure of the corresponding large cardinals notions, that is whenever there is a direct implication from some large cardinal property $A$ to another large cardinal property $B$, then amongst 
the small embeddings witnessing $A$, we find small embeddings witnessing $B$.  For example, we will later show that a cardinal $\kappa$ is inaccessible if and only if for all sufficiently large cardinals $\theta$, there is a small embedding $\map{j}{M}{\HH{\theta}}$ for $\kappa$ with the property that $\crit{j}$ is a strong limit cardinal (see Corollary \ref{corollary:SmallCardsSmallChar}). Hence every small embedding witnessing the supercompactness of a cardinal $\kappa$ with respect to  some sufficiently large cardinal $\theta$ as in Corollary \ref{corollary:SupercompactMagidorStyle} also witnesses the inaccessibility of $\kappa$ with respect to $\theta$.

We will now summarize the contents of our paper. 
In Section \ref{section:smallembeddingcharacterizations}, we will present small embedding characterizations for what we call \emph{Mahlo-like cardinals}, that is notions of large cardinals that are characterized as being stationary limits of certain cardinals, in particular covering the cases of inaccessible and of Mahlo cardinals. 
Section \ref{section:twolemmas} contains  two technical lemmas that will be useful later on. 
In Section \ref{section:indescribables}, we provide small embedding characterizations for \emph{$\Pi^m_n$-indescribable cardinals} for all $0<m,n<\omega$. 
The results of Section \ref{section:subtleineffable} provide such characterizations for \emph{subtle}, for \emph{ineffable}, and for \emph{$\lambda$-ineffable cardinals}. According to Victoria Gitman \cite{MR2830415}, no embedding characterizations of any kind were known so far for these large cardinal notions.
Moreover, these characterizations suggest some variations, and so we introduce the related large cardinal concepts of \emph{supersubtle} and of \emph{$\lambda$-superineffable cardinals}, strengthening the notions of subtle and $\lambda$-ineffable cardinals. We also use the small embedding characterizations of $\lambda$-ineffability and of $\lambda$-superineffability to provide new characterizations of supercompactness. 
In Section \ref{section:filterbased}, we provide small embedding characterizations for various filter based large cardinal notions, that is for \emph{measurable}, for \emph{$\lambda$-supercompact}, and for \emph{$n$-huge cardinals}. 
Section \ref{section:internal} contains a brief introduction to the concept of \emph{internal large cardinals}, which uses small embedding characterizations to describe properties of large cardinals that accessible cardinals can consistently possess. The theory of internal large cardinals will be fully developed in the subsequent paper \cite{hl}. 
In Section \ref{section:internal}, we introduce the internal version of supercompactness with respect to the $\omega_1$-approximation property and use results of Matteo Viale and Christoph Wei{\ss}  to show that this concept is equivalent to a generalized tree property studied in their \cite{MR2838054}.  
We introduce the corresponding internal versions of subtle and of $\lambda$-ineffable cardinals in Section \ref{section:intsubtleineffable}.   
In Section \ref{section:application}, we discuss some problems arising in the consistency proofs of certain generalized tree properties that are presented in \cite{weissthesis} and \cite{MR2959668}. Finally, in Section \ref{section:ConsInternal}, we make use of our concept of internal large cardinals to provide new proofs for these consistency statements, and eliminate the problems discussed in the previous section. We close the paper with some open questions in Section \ref{section:openquestions}.


\section{Mahlo-like cardinals}\label{section:smallembeddingcharacterizations}

In this section, we provide small embedding characterizations for what we call \emph{Mahlo-like} cardinals, that is notions of large cardinals that are characterized as being stationary limits of certain kinds of cardinals.  The following lemma will directly yield these characterizations.

\begin{lemma}\label{lemma:CharStatLimit}
 Given an $\calL_\in$-formula $\varphi(v_0,v_1)$, the following statements are equivalent for every cardinal $\kappa$ and every set $x$: 
 \begin{enumerate}[leftmargin=0.7cm]
  \item $\kappa$ is a regular uncountable cardinal and the set of all ordinals $\lambda<\kappa$ such that $\varphi(\lambda,x)$ holds is stationary in $\kappa$. 
 
  \item For all sufficiently large cardinals $\theta$, there is a small embedding $\map{j}{M}{\HH{\theta}}$ for $\kappa$ with $\varphi(\crit{j},x)$ and $x\in\ran{j}$. 
  
 \end{enumerate}
\end{lemma}

\begin{proof}
   First, assume that (i) holds, and pick a cardinal $\theta>\kappa$ with $x\in\HH{\theta}$. Let $\seq{X_\alpha}{\alpha<\kappa}$ be a continuous and increasing sequence of elementary substructures of $\HH{\theta}$ of cardinality less than $\kappa$ with $x\in X_0$ and $\alpha\subseteq X_\alpha\cap\kappa\in\kappa$ for all $\alpha<\kappa$. By (i), there is an $\alpha<\kappa$ such that $\alpha=X_\alpha\cap\kappa$ and $\varphi(\alpha,x)$ holds. Let $\map{\pi}{X_\alpha}{M}$ denote the corresponding transitive collapse. Then $\map{\pi^{{-}1}}{M}{\HH{\theta}}$ is a small embedding for $\kappa$ with $\varphi(\crit{\pi^{{-}1}},x)$ and $x\in\ran{\pi^{{-}1}}$.

 Now, assume that (ii) holds. Then there is a cardinal $\theta>\kappa$ such that the formula $\varphi$ is absolute between $\HH{\theta}$ and $\VV$, and there is a small embedding $\map{j}{M}{\HH{\theta}}$ for $\kappa$ with the property that $\varphi(\crit{j},x)$ holds and there is a $y\in M$ with $x=j(y)$. Then $\kappa$ is uncountable, because elementarity implies that ${j\restriction(\omega+1)}=\id_{\omega+1}$. Next, assume that $\kappa$ is singular. Then $\crit{j}$ is singular in $M$ and there is a cofinal function $\map{c}{\cof{\crit{j}}^M}{\crit{j}}$ in $M$. In this situation, elementarity implies that $j(c)=c$ is cofinal in $\kappa$, a contradiction. Finally, assume that there is a club $C$ in $\kappa$ such that $\neg\varphi(\lambda,x)$ holds for all $\lambda\in C$. Then elementarity and our choice of $\theta$ imply that, in $M$, there is a club $D$ in $\crit{j}$ such that $\neg\varphi(\lambda,y)$ holds for all $\lambda\in D$. Again, by elementarity and our choice of $\theta$, we know that $j(D)$ is a 
club in $\kappa$ with the property that $\neg\varphi(\lambda,x)$ holds for all $\lambda\in j(D)$. But elementarity also implies that $\crit{j}$ is a limit point of $j(D)$ and therefore $\crit{j}$ is an element of $j(D)$ with $\varphi(\crit{j},x)$, a contradiction.   
\end{proof}

By varying the formula $\varphi$, we can use the above to characterize some of the smallest notions of large cardinals.\footnote{In our below applications, we will not make use of the parameter $x$, i.e.\ $x=\emptyset$.} In fact, we start by showing that we can also characterize regular uncountable cardinals in such a way. Each of the characterizations below is based on a correctness property.\footnote{Note that in general, the characterizations provided by Lemma \ref{lemma:CharStatLimit} are not necessarily correctness properties, as it may not be the case that $M\models\varphi(\crit j,x)$. For example, consider the characterization of a stationary limit of measurable cardinals.}

\begin{corollary}\label{corollary:SmallCardsSmallChar}
 Let $\kappa$ be a cardinal. 
 \begin{enumerate}[leftmargin=0.7cm]
  \item $\kappa$ is uncountable and regular if and only if for all sufficiently large cardinals $\theta$, there is a small embedding $\map{j}{M}{\HH{\theta}}$ for $\kappa$. 
  
  \item $\kappa$ is weakly inaccessible if and only if for all sufficiently large cardinals $\theta$, there is a small embedding $\map{j}{M}{\HH{\theta}}$ for $\kappa$ with the property that $\crit{j}$ is a cardinal. 
  
  \item  $\kappa$ is inaccessible if and only if for all sufficiently large cardinals $\theta$, there is a small embedding $\map{j}{M}{\HH{\theta}}$ for $\kappa$ with the property that $\crit{j}$ is a strong limit cardinal. 

  \item $\kappa$ is weakly Mahlo if and only if for all sufficiently large cardinals $\theta$, there is a small embedding $\map{j}{M}{\HH{\theta}}$ for $\kappa$ with the property that $\crit{j}$ is a regular cardinal. 
  
  \item $\kappa$ is Mahlo if and only if for all sufficiently large cardinals $\theta$, there is a small embedding $\map{j}{M}{\HH{\theta}}$ for $\kappa$ with the property that $\crit{j}$ is an inaccessible cardinal. \qed
 \end{enumerate}
\end{corollary}

\begin{remark}\label{singleremark}
  In many cases, and in particular in each of the above cases, the large cardinal properties in question can also be characterized by the existence of a single elementary embedding. For each of the above, it suffices to require the existence of a single appropriate small embedding $\map{j}{M}{\HH{\kappa^+}}$, as can easily be seen from the proof of Lemma \ref{lemma:CharStatLimit}. For example, a cardinal $\kappa$ is inaccessible if and only if there is a small embedding $\map{j}{M}{\HH{\kappa^+}}$ for $\kappa$ with the property that $\crit{j}$ is a strong limit cardinal. This will in fact be the case for many of the small embedding characterizations that will follow, however we will not make any further mention of this.
\end{remark}

Note that Lemma \ref{lemma:CharStatLimit} implies that small embedding characterizations as in its statement (ii) cannot characterize any notion of large cardinal that implies weak compactness, for weakly compact cardinals satisfy stationary reflection, so for any weakly compact cardinal satisfying (ii), there is in fact a smaller cardinal that satisfies (ii) as well. In the remainder of this paper, we will however provide small embedding characterizations of a different form for many large cardinal notions that imply weak compactness, and in particular also for weak compactness itself.


\section{Two lemmas}\label{section:twolemmas}

Before we continue with further small embedding characterizations, we need to interrupt for the sake of presenting two technical lemmas that will be of use in many places throughout the rest of the paper.

\begin{lemma}\label{lemma:ModelsContainInitialSegment}
 The following statements are equivalent for every small embedding $\map{j}{M}{\HH{\theta}}$ for a cardinal $\kappa$:
 \begin{enumerate}[leftmargin=0.7cm]
  \item $\kappa$ is a strong limit cardinal. 
  
  \item $\crit{j}$ is a strong limit cardinal. 
  
  \item $\crit{j}$ is a cardinal and $\HH{\crit{j}}\subseteq M$. 
  
 \end{enumerate}
\end{lemma}
\begin{proof}
 Assume that (i) holds and pick a cardinal $\nu<\crit{j}$. Since $\crit{j}$ is a strong limit cardinal in $M$, we have $(2^\nu)^M<\crit{j}$. But then $$2^\nu ~ = ~ j((2^\nu)^M) ~ = ~ (2^\nu)^M ~ < ~ \crit{j}$$ and this shows that (ii) holds.  In the other direction, assume (i) fails. By elementarity, there is a cardinal $\nu<\crit{j}$ and an injection of $\crit{j}$ into $\POT{\nu}$ in $M$. Then this injection witnesses that (ii) fails. 
 
 Now, again assume that (i) holds. Then elementarity implies that, in $M$, there is a bijection $\map{s}{\crit{j}}{\HH{\crit{j}}}$ with the property that $\HH{\delta}=s[\delta]$ holds for every strong limit cardinal $\delta<\crit{j}$. 
 Since we already know that (i) implies (ii), we have $\HH{\crit{j}}=j(s)[\crit{j}]$. Fix $x\in\HH{\crit{j}}$ and $\alpha<\crit{j}$ with $j(s)(\alpha)=x$. Since $\crit{j}$ is a strong limit cardinal in $M$, we have ${j\restriction\HH{\crit{j}}^M}=\id_{\HH{\crit{j}}^M}$ and this allows us to conclude that  $x=j(s)(\alpha)=j(s(\alpha))=s(\alpha)\in M$, and hence that (iii) holds. 

Finally, assume for a contradiction that (iii) holds and (i) fails. Then, by elementarity, there is a minimal cardinal $\nu<\crit{j}$ such that either $(2^\nu)^M\geq\crit{j}$ or such that $\POT{\nu}$ does not exist in $M$. By (iii), $\POT{\nu}\subseteq M$. By elementarity, we may pick an injection $\map{\iota}{\crit{j}}{\POT{\nu}}$ in $M$. Define $x=j(\iota)(\crit{j})\in\POT{\nu}\subseteq M$. Then $j(x)=x$, and elementarity yields an ordinal $\gamma<\crit{j}$ with $\iota(\gamma)=x$. But then $j(\iota)(\gamma)=x=j(\iota)(\crit{j})$, contradicting the injectivity of $\iota$.   
\end{proof}

Next, we isolate a certain type of correctness property of small embeddings, for which we will obtain a self-strengthening property in Lemma \ref{lemma:addxtorange} below.

\begin{definition}
 Let $\Phi(v_0,v_1)$ be an $\calL_\in$-formula and let $x$ be a set. We say that the pair $(\Phi,x)$ is \emph{downwards-absolute} if for every cardinal $\kappa$, there is an ordinal $\alpha$ such that $\Phi({j\restriction\HH{\nu}^M},x)$ holds for every small embedding $\map{j}{M}{\HH{\theta}}$ for $\kappa$ with $\Phi(j,x)$ and with $x\in\ran{j}$, and every $\nu>\crit j$ where $\nu$ is a cardinal in $M$ and $j(\nu)>\alpha$.  
\end{definition}

Note that all of the small embedding characterizations provided by Corollary \ref{corollary:SmallCardsSmallChar} use correctness properties which are downwards-absolute, and moreover, all but one of the small embedding characterizations that we are going to derive in the remainder of this paper will use downwards-absolute correctness properties, the exception being the case of subtle cardinals, which are based on what does not seem to be expressible as a correctness property, however we will still have downwards-absoluteness in that case. The verification of downwards-absoluteness will be trivial in each case, and is thus left for the interested reader to check throughout. The following claim and lemma will show why we are interested in the above form of downwards-absoluteness.

\begin{claim}\label{lemma:addxtorangeaux}
 Let $(\Phi,x)$ be downwards-absolute and assume that $\kappa$ is a cardinal with the property that for sufficiently large cardinals $\theta$, there is a small embedding $\map{j}{M}{\HH{\theta}}$ for $\kappa$ with $\Phi(j,x)$ and $x\in\ran{j}$. Then for all sets $z$ and sufficiently large cardinals $\theta$, there is a small embedding $\map{j}{M}{\HH{\theta}}$ for $\kappa$ with $\Phi(j,x)$ and $z\in\ran{j}$. 
\end{claim}

\begin{proof}
 By our assumptions, there is an ordinal $\alpha>\kappa$ such that the following statements hold: 
 \begin{enumerate}
  \item For all cardinals $\theta>\alpha$, there is a small embedding $\map{j}{M}{\HH{\theta}}$ for $\kappa$ with $\Phi(j,x)$ and $x\in\ran{j}$. 
  
  \item If $\map{j}{M}{\HH{\theta}}$ is a small embedding for $\kappa$ such that $\Phi(j,x)$ holds, $\nu>\crit j$ is a cardinal in $M$, $x\in\ran{j}$ and $j(\nu)>\alpha$, then $\Phi({j\restriction\HH{\nu}^M},x)$ holds. 
 \end{enumerate}
 
 Assume for a contradiction that the conclusion of the lemma does not hold. Pick a strong limit cardinal $\theta>\alpha$ with the property that $\HH{\theta}$ is sufficiently absolute in $\VV$ and fix a small embedding $\map{j}{M}{\HH{\theta}}$ for $\kappa$ with the property that $\Phi(j,x)$ holds, and fix $y\in M$ with $j(y)=x$. In this situation, our assumptions, the absoluteness of $\HH{\theta}$ in $\VV$ and the elementarity of $j$ imply that there are $\beta,\vartheta,z\in M$ such that the following statements hold in $M$:
 \begin{enumerate}
  \item[(a)] If $\map{k}{N}{\HH{\eta}}$ is a small embedding for $\crit{j}$ such that $\Phi(k,y)$ holds, $\nu>\crit k$ is a cardinal in $N$, $y\in\ran{k}$ and $k(\nu)>\beta$, then $\Phi({k\restriction\HH{\nu}^N},y)$ holds. 
  \item[(b)] $\vartheta>\beta$ is a cardinal with $y,z\in\HH{\vartheta}$ and there is no small embedding $\map{k}{N}{\HH{\vartheta}}$ for $\crit{j}$ with $\Phi(k,y)$ and $z\in\ran{k}$. 
 \end{enumerate}
 
 By elementarity and our absoluteness assumptions on $\HH{\theta}$, the above implies that the following statements hold in $\VV$: 
 \begin{enumerate}
  \item[$(\rm a)^\prime$] If $\map{k}{N}{\HH{\eta}}$ is a small embedding for $\kappa$ and $\crit{k}<\nu\in N$ is a cardinal in $N$ such that $\Phi(k,x)$ holds, $x\in\ran{k}$ and $k(\nu)>j(\beta)$, then $\Phi({k\restriction\HH{\nu}^N},x)$ holds. 
  
  \item[$(\rm b)^\prime$] $j(\vartheta)>j(\beta)$ is a cardinal with $x,j(z)\in\HH{j(\vartheta)}$ and there is no small embedding $\map{k}{N}{\HH{j(\vartheta)}}$ for $\kappa$ with $\Phi(k,x)$ and $j(z)\in\ran{k}$.  
 \end{enumerate}
 
 Since $j(\vartheta)>j(\beta)$, we can apply the statement $(\rm a)^\prime$ to $\map{j}{M}{\HH{\theta}}$ and $\vartheta$ to conclude that $\Phi(j\restriction\HH{\vartheta}^M,x)$ holds in $\VV$. But we also have $j(z)\in\ran{j\restriction\HH{\vartheta}^M}$ and together these statements contradict $(\rm b)^\prime$. 
\end{proof}

We now show that the above lemma implies a somewhat stronger statement that essentially  allows us to switch the quantifiers on $z$ and on $\theta$ in the statement of the lemma.

\begin{lemma}\label{lemma:addxtorange}
 Let $(\Phi,x)$ be downwards-absolute and assume that $\kappa$ is a cardinal with the property that for sufficiently large cardinals $\theta$, there is a small embedding $\map{j}{M}{\HH{\theta}}$ for $\kappa$ with $\Phi(j,x)$ and $x\in\ran{j}$. Then for all sufficiently large cardinals $\theta$ and for all $z\in\HH{\theta}$, there is a small embedding $\map{j}{M}{\HH{\theta}}$ for $\kappa$ with $\Phi(j,x)$ and $z\in\ran{j}$. 
\end{lemma}

\begin{proof}
  Fix a sufficiently large cardinal $ \theta$ and some $z\in\HH{\theta}$. By Claim \ref{lemma:addxtorangeaux}, there is a cardinal $\theta'$ and a small embedding $\map{j'}{M'}{\HH{\theta'}}$ for $\kappa$ with $\Phi(j',x)$ and $z,\theta\in\ran{j'}$. Let $j$ be the restriction of $j'$ to $M=\HH{(j')^{-1}(\theta)}^{M'}$. Then $\map{j}{M}{\HH{\theta}}$ is a small embedding for $\kappa$ with $\Phi(j,x)$ and $z\in\ran{j}$.
\end{proof}


\section{Indescribable Cardinals}\label{section:indescribables}

In this section, we provide small embedding characterizations for \emph{indescribable} cardinals. Recall that, given $0<m,n<\omega$, a cardinal $\kappa$ is \emph{$\Pi^m_n$-indescribable} if for every $\Pi^m_n$-formula $\varphi(A_0,\ldots,A_{n-1})$ whose parameters $A_0,\ldots,A_{n-1}$ are subsets of $\VV_\kappa$, the assumption $\VV_\kappa\models\varphi(A_0,\ldots,A_{n-1})$ implies that there is a $\delta<\kappa$ such that $\VV_\delta\models\varphi(A_0\cap\VV_\delta,\ldots,A_{n-1}\cap\VV_\delta)$. Moreover, remember that, given an uncountable cardinal $\kappa$, a transitive set $M$ of cardinality $\kappa$ is a \emph{$\kappa$-model} if $\kappa\in M$, ${}^{{<}\kappa}M\subseteq M$ and $M$ is a model of $\ZFC^-$.  Our small embedding characterizations of indescribable cardinals build on the following embedding characterizations of these cardinals by Kai Hauser (see \cite{MR1133077}).

\begin{theorem}[{\cite[Theorem 1.3]{MR1133077}}]\label{theorem:Hauser}
  The following statements are equivalent for every inaccessible cardinal $\kappa$ and all $0<m,n<\omega$:  

  \begin{enumerate}[leftmargin=0.7cm]
    \item $\kappa$ is $\Pi^m_n$-indescribable.

    \item For every $\kappa$-model $M$,
    there is a transitive set $N$ and an elementary embedding $\map{j}{M}{N}$  with $\crit{j}=\kappa$ such that the following statements hold:
  \begin{enumerate}
   \item $N$ has cardinality $\beth_{m-1}(\kappa)$, ${}^{{<}\kappa}N\subseteq N$ and $j,M\in N$. 

     \item If $m>1$, then ${}^{\beth_{m-2}(\kappa)}N\subseteq N$. 

  \item We have $$\VV_\kappa\models\varphi ~ \Longleftrightarrow ~ (\VV_\kappa\models\varphi)^N$$ for all $\Pi^m_{n-1}$-formulas $\varphi$  whose parameters are contained in $N\cap\VV_{\kappa+m}$. 
  \end{enumerate} 
  \end{enumerate}
\end{theorem}

Note that, in case $m>1$, the statement $j,M\in N$ in (a) is a direct consequence of (b). It is not explicitly mentioned, but easy to observe from the proof given in \cite{MR1133077} that this can also be equivalently required in case $m=1$ (for weakly compact cardinals, this is in fact what became known as their \emph{Hauser characterization}). The above theorem allows us to characterize indescribable cardinals through small embeddings, in two ways.

\begin{lemma}\label{lemma:SmallEmbCharIndescribable}
 Given $0<m,n<\omega$, the following statements are equivalent for every cardinal $\kappa$: 
 \begin{enumerate}[leftmargin=0.7cm]
  \item $\kappa$ is $\Pi^m_n$-indescribable.
  
  \item For all sufficiently large cardinals $\theta$, there is a small embedding $\map{j}{M}{\HH{\theta}}$ for $\kappa$ with the property that 
$$(\VV_{\crit{j}}\models\varphi)^M ~ \Longrightarrow ~ \VV_{\crit{j}}\models\varphi$$ for 
  every $\Pi^m_n$-formula $\varphi$ whose parameters are contained in $M\cap\VV_{\crit{j}+1}$.  

  \item For all sufficiently large cardinals $\theta$ and all $x\in \VV_{\kappa+1}$, there is a small embedding $\map{j}{M}{\HH{\theta}}$ for $\kappa$ with $x\in\ran j$ and with the property that
  $$(\VV_{\crit{j}}\models\varphi)^M ~ \Longrightarrow ~ \VV_{\crit{j}}\models\varphi$$ for 
  every $\Pi^m_n$-formula $\varphi$ using only $j^{{-}1}(x)$ as a parameter.  
  \end{enumerate}
\end{lemma}

\begin{proof}
 First, assume that (i) holds. Pick a cardinal $\theta>\beth_m(\kappa)$ and a regular cardinal $\vartheta>\theta$ with $\HH{\theta}\in\HH{\vartheta}$. Since $\kappa$ is inaccessible, there is an elementary submodel $X$ of $\HH{\vartheta}$ of cardinality $\kappa$ with $\kappa+1\cup\{\theta\}\subseteq X$ and ${}^{{<}\kappa}X\subseteq X$. Let $\map{\pi}{X}{M}$ denote the corresponding transitive collapse. Then $M$ is a $\kappa$-model and Theorem \ref{theorem:Hauser} yields an elementary embedding $\map{j}{M}{N}$ with $\crit{j}=\kappa$ that satisfies the properties (a)--(c) listed in the statement (ii) of Theorem \ref{theorem:Hauser}. Note that the assumption ${}^{{<}\kappa}N\subseteq N$ implies that $\kappa$ is inaccessible in $N$.

 \begin{claim*}
  We have $$(\VV_\kappa\models\varphi)^M ~ \Longrightarrow ~ (\VV_\kappa\models\varphi)^N$$ for all $\Pi^m_n$-formulas $\varphi$ whose parameters are contained in $M\cap\VV_{\kappa+1}$.   
 \end{claim*}
 
 \begin{proof}[Proof of the Claim]
 
  Assume that $(\VV_\kappa\models\varphi)^M$ holds. This assumption implies that $\VV_\kappa\models\varphi$ holds, because $\pi^{{-}1}\restriction\VV_{\kappa+1}=\id_{\VV_{\kappa+1}}$ and $\VV_{\kappa+m}\in\HH{\vartheta}$. 
By Statement (c) of Theorem \ref{theorem:Hauser}, we can conclude that $(\VV_\kappa\models\varphi)^N$ holds. 
 \end{proof}

Set $\theta_*=\pi(\theta)$, $M_*=\HH{\theta_*}^M$ and $j_*=j\restriction M_*$. Since $j,M\in N$, we also have $j_*,M_*\in N$. Moreover, in $N$, the map $\map{j_*}{M_*}{\HH{j(\theta_*)}^N}$ is a small embedding for $j(\kappa)$. 
If $\varphi$ is a $\Pi^m_n$-formula  with parameters in $M_*\cap\VV_{\kappa+1}$ such that $(\VV_\kappa\models\varphi)^{M_*}$ holds, then $\theta>\beth_m(\kappa)$ implies that $(\VV_\kappa\models\varphi)^M$ holds, and we can use the above claim to conclude that $(\VV_\kappa\models\varphi)^N$ holds. By elementarity, this shows that, in $M$, there is a small embedding $\map{j^\prime}{M^\prime}{\HH{\theta_*}}$ for $\kappa$ such that $\crit{j^\prime}$ is inaccessible and $\VV_{\crit{j^\prime}}\models\varphi$ holds for every $\Pi^m_n$-formula $\varphi$ with parameters in $M^\prime\cap\VV_{\crit{j}+1}$ with the property that $(\VV_{\crit{j}}\models\varphi)^{M^\prime}$ holds. Since $\VV_{\kappa+m}\in\HH{\vartheta}$, we can conclude that $\pi^{{-}1}(j^\prime)$ is a small embedding for $\kappa$ witnessing that (ii) holds for $\theta$.

Next, Lemma \ref{lemma:addxtorange} shows that (iii) is a consequence of (ii). Hence,  assume, towards a contradiction, that (iii) holds and that there is a $\Pi^m_n$-formula $\varphi(x)$ with $x\in\VV_{\kappa+1}$, $\VV_\kappa\models\varphi(x)$ and $\VV_\delta\models\neg\varphi(x\cap\VV_\delta)$  for all  $\delta<\kappa$. Pick a regular cardinal $\theta>\beth_m(\kappa)$ such that there is a small embedding $\map{j}{M}{\HH{\theta}}$ for $\kappa$ that satisfies the statements  listed in (iii) with respect to $x$. Since $\VV_{\kappa+m}\in\HH{\theta}$, elementarity yields that $(\VV_{\crit{j}}\models\varphi(j^{-1}(x)))^M$. Thus our assumptions on $j$ allow us to conclude that $\VV_{\crit{j}}\models\varphi(j^{-1}(x))$, contradicting the above assumption.  
\end{proof}

In the case $m=1$, the equivalence between the statements (i) and (ii) in Lemma \ref{lemma:SmallEmbCharIndescribable} can be rewritten in the following way, using the fact that we can canonically identify  $\Sigma_n$-formulas using  parameters in $\HH{\crit j^+}$ with $\Sigma^1_n$-formulas using parameters in $V_{\crit j+1}$ such that the given $\Sigma_n$-formula holds true in $\HH{\crit j^+}$ if and only if the corresponding  $\Sigma^1_n$-formula holds in $V_{\crit j}$.

\begin{corollary}
  Given $0<n<\omega$, the following statements are equivalent for every cardinal $\kappa$: 
 \begin{enumerate}[leftmargin=0.7cm]
  \item $\kappa$ is $\Pi^1_n$-indescribable.

  \item For all sufficiently large cardinals $\theta$, there is a small embedding $\map{j}{M}{\HH{\theta}}$ for $\kappa$ such that $\HH{\crit{j}^+}^M\prec_{\Sigma_n}\HH{\crit{j}^+}$. \qed 
 \end{enumerate}
\end{corollary}

As mentioned in the introduction, throughout this paper, we will show that whenever we have a direct implication between two large cardinals properties that we provide small embedding characterizations for, then amongst the embeddings witnessing the stronger property, we may also find such witnessing the weaker one. First of all,  Lemma \ref{lemma:SmallEmbCharIndescribable} directly shows that small embeddings witnessing $\Pi^m_n$-indescribability also witnesses all smaller degrees of indescribability. Next, it is also easy to see that these embeddings possess the properties mentioned in the small embedding characterization of Mahlo cardinals provided by Corollary \ref{corollary:SmallCardsSmallChar}.

\begin{corollary}
 Given $0<m,n<\omega$, let $\kappa$ be $\Pi^m_n$-indescribable and let $\theta$ be a sufficiently large cardinal such that there is a small embedding $\map{j}{M}{\HH{\theta}}$ for $\kappa$ witnessing $\Pi^m_n$-indescribability of $\kappa$, as in statement (ii) of Lemma \ref{lemma:SmallEmbCharIndescribable}. Then $\crit j$ is inaccessible, hence $j$ witnesses the Mahloness of $\kappa$, as in statement (v) of Corollary \ref{corollary:SmallCardsSmallChar}.   \qed 
\end{corollary}


\section{Subtle, Ineffable and $\lambda$-Ineffable Cardinals}\label{section:subtleineffable}

The results of this section provide small embedding characterizations for \emph{ineffable} and \emph{subtle} cardinals (introduced in \cite{JensenKunen1969:Ineffable}) and \emph{$\lambda$-ineffable} cardinals (introduced in \cite{MR0327518}). These large cardinal concepts all rely on the following definition.

\begin{definition}
 Given a set $A$, a sequence $\seq{d_a}{a\in A}$ is an \emph{$A$-list} if $d_a\subseteq a$ holds for all $a\in A$. 
\end{definition}

Then an uncountable regular cardinal $\kappa$ is \emph{subtle} if for every $\kappa$-list $\seq{d_\alpha}{\alpha<\kappa}$ and every club $C$ in $\kappa$, there are $\alpha,\beta\in C$ with $\alpha<\beta$ and $d_\alpha=d_\beta\cap\alpha$.

\begin{lemma}\label{lemma:SubtleSmallChar}
 The following statements are equivalent for every cardinal $\kappa$: 
 \begin{enumerate}[leftmargin=0.7cm]
  \item $\kappa$ is subtle.

  \item For all sufficiently large cardinals $\theta$, for every $\kappa$-list $\vec{d}=\seq{d_\alpha}{\alpha<\kappa}$ and for every club $C$ in $\kappa$,  there is a small embedding $\map{j}{M}{\HH{\theta}}$ for $\kappa$ such that $\vec{d},C\in \ran{j}$ 
   and $d_\alpha=d_{\crit{j}}\cap\alpha$ for some $\alpha\in C\cap\crit{j}$.
 \end{enumerate} 
\end{lemma}

\begin{proof}
  First, assume first that $\kappa$ is subtle. Pick a cardinal $\theta>\kappa$, a club $C$ in $\kappa$ and a $\kappa$-list $\vec{d}=\seq{d_\alpha}{\alpha<\kappa}$. Let $\seq{X_\alpha}{\alpha<\kappa}$ be a continuous and increasing sequence of elementary substructures of $\HH{\theta}$ of cardinality less than $\kappa$ with $\vec{d},C\in X_0$ and $\alpha\subseteq X_\alpha\cap\kappa\in\kappa$ for all $\alpha<\kappa$. Set $D=\Set{\alpha\in C}{\alpha=M_\alpha\cap\kappa}$. Then $D$ is a club in $\kappa$ and the subtlety of $\kappa$ yields $\alpha,\beta\in D\subseteq C$ with $\alpha<\beta$ and $d_\alpha=d_\beta\cap\alpha$. Let $\map{\pi}{X_\beta}{M}$ denote the transitive collapse of $X_\beta$. Then  $\map{\pi^{{-}1}}{M}{\HH{\theta}}$ is a small embedding for $\kappa$ with $\crit{\pi^{{-}1}}=\beta$, $\vec{d},C\in\ran{\pi^{{-}1}}$ and $d_\alpha=d_{\crit{\pi^{{-}1}}}\cap\alpha$. 
  
Now, assume that (ii) holds. Then Corollary \ref{corollary:SmallCardsSmallChar} implies that $\kappa$ is uncountable and regular. Fix a $\kappa$-list $\vec{d}=\seq{d_\alpha}{\alpha<\kappa}$ and a club $C$ in $\kappa$. Let $\theta$ be a sufficiently large cardinal  such that there is a small embedding $\map{j}{M}{\HH{\theta}}$ for $\kappa$ such that $\vec{d},C\in \ran{j}$ and $d_\alpha=d_{\crit{j}}\cap\alpha$ for some $\alpha\in C\cap\crit{j}$. Since $C\in\ran{j}$, elementarity implies that $\crit{j}$ is a limit point of $C$ and hence $\crit{j}\in C$.  
\end{proof}

\begin{remark}
 Note that, unlike all the other small embedding characterizations that we provide in this paper, the above characterization of subtle cardinals is not based on a correctness property between the domain model $M$ and $\VV$. However, we think that the above characterization is still useful. This view will be undermined by the applications of this characterization presented in Section \ref{section:ConsInternal}.
\end{remark}

In \cite[Theorem 3.6.3]{MR0460120}, it is shown that there is a totally indescribable cardinal below any subtle cardinal, and in fact a minor adaption of the proof of that theorem shows that there is a stationary set of totally indescribable cardinals below every subtle cardinal. Not every small embedding for $\kappa$ witnessing an instance of subtleness has a critical point that is totally indescribable (and would thus witness that $\kappa$ is a stationary limit of totally indescribable cardinals by Lemma \ref{lemma:CharStatLimit}), since if the $\kappa$-list $\vec d$ and the club $C$ are both trivial, an embedding witnessing the corresponding instance of subtleness as in (ii) of Lemma \ref{lemma:SubtleSmallChar} merely witnesses the regularity of $\kappa$, by Corollary \ref{corollary:SmallCardsSmallChar}. However, the next lemma shows that we can pick $\vec d$ and $C$ such that any small embedding witnessing subtleness of $\kappa$ with respect to  $\vec d$ and $C$ has a critical point that is totally indescribable.  

\begin{lemma}\label{Lemma:CritIndes}
    Let $\kappa$ be a subtle cardinal. 
   Then there is a $\kappa$-list $\vec{d}$ and a club $C$ in $\kappa$ with the property that, whenever $\theta$ is a  sufficiently large cardinal and $\map{j}{M}{\HH{\theta}}$ is a small embedding for $\kappa$ witnessing the subtlety of $\kappa$ with respect to $\vec d$ and $C$, as in statement (ii) of Lemma \ref{lemma:SubtleSmallChar}, then $\crit j$ is totally indescribable, hence $j$ witnesses that $\kappa$ is a stationary limit of totally indescribable cardinals, as in statement (ii) of Lemma \ref{lemma:CharStatLimit}. 
 \end{lemma}

\begin{proof}
 Let $C$ be the club $\Set{\alpha<\kappa}{\betrag{\VV_\alpha}=\alpha}$ and let $\map{h}{\VV_\kappa}{\kappa}$ be a bijection with $h[\VV_\alpha]=\alpha$ for all $\alpha\in C$. Let $\goedel{\cdot}{\cdot}$ denote the G\"odel pairing function and let $\vec{d}=\seq{d_\alpha}{\alpha<\kappa}$ be a $\kappa$-list with the following properties: 
  \begin{enumerate}
   \item If $\alpha\in C$ is not totally indescribable, then there is a $\Pi^m_n$-formula $\varphi$ and a subset $A$ of $\VV_\alpha$ such that these objects provide a counterexample to the $\Pi^m_n$-indescribability of $\alpha$. Then $d_\alpha=\{\goedel{0}{\lceil\varphi\rceil}\}\cup\Set{\goedel{1}{h(a)}}{a\in A}$, where $\lceil\varphi\rceil\in\omega$ is the G\"odel number of $\varphi$ in some fixed G\"odelization of second order set theory. 
   
   \item Otherwise, $d_\alpha$ is the empty set. 
  \end{enumerate}

Let $\theta$ be a sufficiently large cardinal and let $\map{j}{M}{\HH{\theta}}$ be a small embedding for $\kappa$ that witnesses the subtlety of $\kappa$ with respect to $\vec d$ and $C$, as in Lemma \ref{lemma:SubtleSmallChar}. Then $\crit{j}\in C$.  Assume for a contradiction that $\crit j$ is not totally indescribable. Then there is a $\Pi^m_n$-formula $\varphi$ and a subset $A$ of $\VV_\alpha$ such that $d_\alpha=\{\goedel{0}{\lceil\varphi\rceil}\}\cup\Set{\goedel{1}{h(a)}}{a\in A}$, $\VV_{\crit{j}}\models\varphi(A)$ and $\VV_\alpha\models\neg\varphi(A\cap\VV_\alpha)$ for all $\alpha<\crit{j}$. By our assumptions, there is an $\alpha\in C\cap \crit j$ with $d_\alpha=d_{\crit j}\cap\alpha$. In this situation, our definition of $d_\alpha$ ensures that the formula $\varphi$ and the subset $A\cap\VV_\alpha$ of $\VV_\alpha$ provide a counterexample to the $\Pi^m_n$-indescribability of $\alpha$. In particular, we know that $\VV_\alpha\models\varphi(A\cap\VV_\alpha)$ holds, a contradiction.  
\end{proof}

Next, we consider small embedding characterizations of ineffable cardinals, where a regular uncountable cardinal $\kappa$ is \emph{ineffable} if for every $\kappa$-list $\seq{d_\alpha}{\alpha<\kappa}$, there exists a subset $D$ of $\kappa$ such that the set $\Set{\alpha<\kappa}{d_\alpha=D\cap\alpha}$ is stationary in $\kappa$.

\begin{lemma}\label{lemma:IneffableSmallChar}
 The following statements are equivalent for every cardinal $\kappa$: 
 \begin{enumerate}[leftmargin=0.7cm]
  \item $\kappa$ is ineffable.

  \item For all sufficiently large cardinals $\theta$ and for every $\kappa$-list $\vec{d}=\seq{d_\alpha}{\alpha<\kappa}$, there is a small embedding $\map{j}{M}{\HH{\theta}}$ for $\kappa$ with $\vec{d}\in\ran{j}$ and $d_{\crit{j}}\in M$. 
 \end{enumerate}
\end{lemma}

\begin{proof}
  Assume first that $\kappa$ is ineffable. Pick a $\kappa$-list $\vec{d}=\seq{d_\alpha}{\alpha<\kappa}$ and a cardinal $\theta>\kappa$. Using the ineffability of $\kappa$, we find a subset $D$ of $\kappa$ such that the set $S=\Set{\alpha<\kappa}{d_\alpha=D\cap\alpha}$ is stationary in $\kappa$. With the help of a continuous chain of elementary submodels of $\HH{\theta}$, we then find $X\prec \HH{\theta}$ of size less than $\kappa$ such that $\vec{d},D\in X$ and $X\cap\kappa\in S$. Let $\map{\pi}{X}{M}$ denote the corresponding transitive collapse. Then $\map{\pi^{{-}1}}{M}{\HH{\theta}}$ is a small embedding for $\kappa$ with $\crit{j}\in S$, $\vec{d}\in\ran{\pi^{{-}1}}$ and $d_{\crit j}=D\cap\crit{j}=\pi(D)\in M$. 

 Assume now that (ii) holds. Let $\vec{d}=\seq{d_\alpha}{\alpha<\kappa}$ be a $\kappa$-list and let $\theta$ be a sufficiently large cardinal such that there exists a small embedding $\map{j}{M}{\HH{\theta}}$ for $\kappa$ with $\vec{d}\in\ran{j}$ and $d_{\crit{j}}\in M$.  Assume that there is a club $C$ in $\kappa$ with $d_\alpha\neq j(d_{\crit{j}})\cap\alpha$ for all $\alpha\in C$. Since $\vec{d}\in\ran{j}$, elementarity implies that there is a club subset $C_0$ of $\crit{j}$ in $M$ with $d_\alpha\neq j(d_{\crit{j}})\cap\alpha$ for all $\alpha\in j(C_0)$. But $j(C_0)$ is a club in $\kappa$ and elementarity implies that $\crit{j}$ is a limit point of $j(C_0)$ with $d_{\crit{j}}=j(d_\crit{j})\cap\crit{j}$, a contradiction.  This argument shows that the set $\Set{\alpha<\kappa}{d_\alpha=j(d_{\crit{j}})\cap\alpha}$ is stationary in $\kappa$. 
\end{proof}

Small embeddings for $\kappa$ witnessing that $\kappa$ is ineffable also witness that $\kappa$ is subtle.

\begin{lemma}\label{lemma:IneffableWitnessSubtle}
  Let $\kappa$ be ineffable, let $\vec{d}$ be a $\kappa$-list and let $C$ be a club in $\kappa$.  If $\theta$ is a  sufficiently large cardinal such that there is a small embedding  $\map{j}{M}{\HH{\theta}}$ for $\kappa$ witnessing the ineffability of $\kappa$ with respect to $\vec d$, as in statement (ii) of Lemma \ref{lemma:IneffableSmallChar}, then $C\in\ran{j}$\footnote{The assumption that $C$ is contained in the range of $j$ is harmless by Lemma \ref{lemma:addxtorange}.} implies that $j$ also witnesses the subtlety of $\kappa$ with respect to $\vec d$ and $C$, as in statement (ii) of Lemma \ref{lemma:SubtleSmallChar}.  
\end{lemma}

\begin{proof}
 Pick a club subset $C_0$ of $\crit{j}$ in $M$ with $j(C_0)=C$. 
  Since $\crit{j}$ is an element of $j(C_0)$ with $d_{\crit{j}}=j(d_{\crit{j}})\cap\crit{j}$, elementarity implies that there is an $\alpha\in C_0\cap\crit{j}$ with $d_\alpha=d_{\crit{j}}\cap\alpha$. Then  $\alpha$ is an element of $C\cap\crit{j}$ with  $d_\alpha=j(d_{\crit{j}})\cap\alpha=d_{\crit{j}}\cap\alpha$.  
\end{proof}

We now show that small embeddings for $\kappa$ witnessing that $\kappa$ is ineffable also witness that $\kappa$ is $\Pi^1_2$-indescribable. Note that the least ineffable cardinal is not $\Pi^1_3$-indescribable.

\begin{lemma}
  Let $\kappa$ be ineffable and let $x\in\VV_{\kappa+1}$. 
  Then there is a $\kappa$-list $\vec{d}$ and a subset $h$ of $\VV_\kappa$ with the property that whenever $\theta$ is a sufficiently large cardinal and $\map{j}{M}{\HH{\theta}}$ is a small embedding for $\kappa$ witnessing the ineffability of $\kappa$ with respect to $\vec{d}$, as in statement (ii) of Lemma \ref{lemma:IneffableSmallChar}, then $h,x\in\ran{j}$ implies that $j$ witnesses the $\Pi^1_2$-indescribability of $\kappa$ with respect to $x$, as in statement (iii) of Lemma \ref{lemma:SmallEmbCharIndescribable}. 
\end{lemma}

\begin{proof}
Fix a bijection $\map{h}{\VV_\kappa}{\kappa}$ with $h[\VV_\alpha]=\alpha$ for every strong limit cardinal $\alpha<\kappa$ and a $\kappa$-list $\vec{d}=\seq{d_\alpha}{\alpha<\kappa}$ such that the following statements hold for all $\alpha<\kappa$: 
 \begin{enumerate}
  \item If $\alpha$ is  inaccessible, then $d_\alpha\neq\emptyset$ if and only if there is a $\Sigma^1_1$-formula $\psi_\alpha(v_0,v_1)$ and $\emptyset\ne y_\alpha\in\VV_{\alpha+1}$ with the property that $\VV_\kappa\models\forall Z ~ \psi_\alpha(x,Z)$ and $\VV_\alpha\models\neg\psi_\alpha(x\cap\VV_\alpha,y_\alpha)$. We let $d_\alpha=h[y_\alpha]$ in this case. 
  
  \item If $\alpha$ is a singular cardinal, then $d_\alpha$ is a cofinal subset of $\alpha$ of order-type $\cof{\alpha}$. 
  
  \item Otherwise, $d_\alpha=\emptyset$. 
 \end{enumerate}

Let $\theta$ be a sufficiently large cardinal and let $\map{j}{M}{\HH{\theta}}$ be a small embedding for $\kappa$ with $\vec{d},h,x\in\ran{j}$ and $d_{\crit{j}}\in M$. Then Lemma \ref{lemma:ModelsContainInitialSegment} implies that $\crit{j}$ is a strong limit cardinal. Since $\crit{j}$ is regular in $M$, our definition of $\vec{d}$ ensures that $\crit{j}$ is inaccessible. Then our assumptions imply that $j^{{-}1}(x)=x\cap\VV_{\crit{j}}\in M$.  

Assume that there is a $\Pi^1_2$-formula $\varphi(v)$ with $(\VV_\crit{j}\models\varphi(x\cap\VV_{\crit{j}}))^M$ and $\VV_{\crit{j}}\models\neg\varphi(x\cap\VV_{\crit{j}})$. Then elementarity implies that $\VV_\kappa\models\varphi(x)$ holds, and this allows us to conclude that the set $d_{\crit{j}}$ is not empty, $\VV_\kappa\models\forall Z ~ \psi_{\crit{j}}(x,Z)$ and $\VV_{\crit{j}}\models\neg\psi_{\crit{j}}(x\cap\VV_{\crit{j}},y_{\crit{j}})$. Since $d_{\crit{j}}\in M$ and $h\in\ran{j}$, we obtain that $y_{\crit{j}}\in M$, and elementarity implies that $(\VV_\crit{j}\models\psi_{\crit{j}}(x\cap\VV_{\crit{j}},y_{\crit{j}}))^M$. Then Lemma \ref{lemma:ModelsContainInitialSegment} shows that $\VV_{\crit{j}}\subseteq M$ and we can apply $\Sigma^1_1$-upwards absoluteness to conclude that $\VV_{\crit{j}}\models\psi_{\crit{j}}(x\cap\VV_{\crit{j}},y_{\crit{j}})$, a contradiction.  
\end{proof}

The above small embedding characterization of ineffable cardinals can easily be modified to produce such characterizations for the generalizations of ineffability studied by Menachem Magidor in \cite{MR0327518}. Remember that, given a regular uncountable cardinal $\kappa$ and a cardinal $\lambda\geq\kappa$, the cardinal $\kappa$ is \emph{$\lambda$-ineffable} if for every $\Poti{\lambda}{\kappa}$-list $\vec{d}=\seq{d_a}{a\in\Poti{\lambda}{\kappa}}$, there exists a subset $D$ of $\lambda$ such that the set $\Set{a\in\Poti{\lambda}{\kappa}}{d_a=D\cap a}$ is stationary in $\Poti{\lambda}{\kappa}$. Since $\kappa$ is a club in $\Poti{\kappa}{\kappa}$ for every uncountable regular cardinal $\kappa$, it is easy to see that a cardinal $\kappa$ is ineffable if and only if it is $\kappa$-ineffable.  The small embedding characterization of ineffability provided by Lemma \ref{lemma:IneffableSmallChar} now generalizes to $\lambda$-ineffability as follows. Note that requiring $\delta<\kappa$ below should be seen as a smallness requirement of the domain model $M$ of the embedding. It can be read off from the proof below that we could equivalently require that $|M|$ be less than $\kappa$.

\begin{lemma}\label{lemma:LambdaIneffableSmallChar}
The following statements are equivalent for all cardinals $\kappa\leq\lambda$: 
  \begin{enumerate}[leftmargin=0.7cm]
    \item $\kappa$ is $\lambda$-ineffable.
   

    \item For all sufficiently large cardinals $\theta$ and every $\Poti{\lambda}{\kappa}$-list $\vec{d}=\seq{d_a}{a\in\Poti{\lambda}{\kappa}}$, there is a small embedding $\map{j}{M}{\HH{\theta}}$ for $\kappa$ and $\delta\in M\cap\kappa$ such that $j(\delta)=\lambda$, $\vec{d}\in\ran{j}$ and $j^{{-}1}[d_{j[\delta]}]\in M$.
\end{enumerate}
\end{lemma}

\begin{proof}
 Assume first that $\kappa$ is $\lambda$-ineffable. Fix a $\Poti{\lambda}{\kappa}$-list $\vec{d}=\seq{d_a}{a\in\Poti{\lambda}{\kappa}}$ and a cardinal $\theta$ with $\Poti{\lambda}{\kappa}\in\HH{\theta}$. Then the $\lambda$-ineffability of $\kappa$ yields a subset $D$ of $\lambda$  such that the $S=\Set{a\in\Poti{\kappa}{\lambda}}{d_a=D\cap a}$ is stationary in $\Poti{\lambda}{\kappa}$. In this situation, we can find $X\prec\HH{\theta}$ of cardinality less than $\kappa$ such that $\vec{d},D\in X$, $X\cap\kappa\in\kappa$ and $X\cap\lambda\in S$. Let $\map{\pi}{X}{M}$ denote the corresponding transitive collapse. Then $\pi(\lambda)<\kappa$ and $\map{\pi^{{-}1}}{M}{\HH{\theta}}$ is a small embedding for $\kappa$ with $\vec{d}\in\ran{\pi^{{-}1}}$. Moreover, we have $$\pi[d_{\pi^{{-}1}[\pi(\lambda)]}] ~ = ~ \pi[d_{X\cap\lambda}] ~ = ~ \pi[D\cap X] ~ = ~ \pi(D) ~ \in M.$$ 
 
Now, assume that (ii) holds, and let $\vec{d}=\seq{d_a}{a\in\Poti{\lambda}{\kappa}}$ be a $\Poti{\lambda}{\kappa}$-list. Pick a small embedding $\map{j}{M}{\HH{\theta}}$ for $\kappa$ and $\delta\in M\cap\kappa$ with $j(\delta)=\lambda$, $\vec{d}\in\ran{j}$ and $d=j^{{-}1}[d_{j[\delta]}]\in\POT{\delta}^M$. We define $S=\Set{a\in\Poti{\lambda}{\kappa}}{d_a=j(d)\cap a}\in\ran{j}$. Assume for a contradiction that the set $S$ is not stationary in $\Poti{\lambda}{\kappa}$. Then there is a function $\map{f}{\Poti{\lambda}{\omega}}{\Poti{\lambda}{\kappa}}$ with $\clo{f}\cap S=\emptyset$, where $\clo{f}$ denotes the set of all $a\in\Poti{\lambda}{\kappa}$ with $f(b)\subseteq a$ for all $b\in\Poti{a}{\omega}$. Since $S\in\ran{j}$, elementarity yields a function $\map{f_0}{\Poti{\delta}{\omega}}{\Poti{\delta}{\crit{j}}}$ in $M$ with $\clo{j(f_0)}\cap S=\emptyset$. Pick $b\in\Poti{j[\delta]}{\omega}$. Then $b\in\ran j$, and hence $j^{{-}1}(b)=j^{{-}1}[b]\in M$, and there is $a\in\rm{Cl}_{f_0}^M$ with $j^{{-}1}[b]\subseteq a\in\Poti{\delta}{\crit{j}}^M$. In this situation, we have $j(f_0)(b)=j(f_0(j^{{-}1}[b]))\subseteq j(a)=j[a]\subseteq j[\delta]$. These computations show that $j[\delta]\in\clo{j(f_0)}$. But we also have $j(d)\cap j[\delta]=j[d]=d_{j[\delta]}$, and this shows that $j[\delta]\in\clo{j(f_0)}\cap S$, a contradiction. 
\end{proof}

It is easy to see that small embeddings witnessing certain degrees of ineffability also witness all smaller degrees.

\begin{proposition}\label{proposition:DownwardsLambdaIneffable}
 Let $\kappa$ be a $\lambda$-ineffable cardinal, let $\kappa\leq\lambda_0<\lambda$ be a cardinal and let $\vec{d}=\seq{d_a}{a\in\Poti{\lambda_0}{\kappa}}$ be a $\Poti{\lambda_0}{\kappa}$-list.  If $\theta$ is a  sufficiently large cardinal such that there is a small embedding $\map{j}{M}{\HH{\theta}}$ for $\kappa$ witnessing the $\lambda$-ineffability of $\kappa$ with respect to $\seq{d_{a\cap\lambda_0}}{a\in\Poti{\lambda}{\kappa}}$, as in statement (ii) of Lemma \ref{lemma:LambdaIneffableSmallChar}, then $\vec{d}\in\ran{j}$ implies that $j$ also witnesses the $\lambda_0$-ineffability of $\kappa$ with respect to $\vec{d}$ in this way. \qed
\end{proposition}

The following result uses ideas from \cite{MR2838054} and \cite{MR2959668} to derive a strengthening of Lemma \ref{lemma:ModelsContainInitialSegment} for many small embeddings witnessing $\lambda$-ineffability, that we will make use of in Section \ref{section:ConsInternal} below.

\begin{lemma}\label{lemma:LambdaIneffableClosureProperties}
 Let $\kappa$ be a $\lambda$-ineffable cardinal. If $\lambda=\lambda^{{<}\kappa}$, then there is a $\Poti{\lambda}{\kappa}$-list $\vec{d}$ and a set $x$ with the property that whenever $\theta$ is a sufficiently large cardinal such that there is a small embedding $\map{j}{M}{\HH{\theta}}$ for $\kappa$ and $\delta\in M\cap\kappa$ witnessing the $\lambda$-ineffability of $\kappa$ with respect to $\vec{d}$, as in statement (ii) of Lemma \ref{lemma:LambdaIneffableSmallChar}, then $x\in\ran{j}$ implies that $\crit j$ is an inaccessible cardinal and $\Poti{\delta}{\crit{j}}\subseteq M$. 
\end{lemma}

\begin{proof}
 Fix a bijection $\map{f}{\Poti{\lambda}{\kappa}}{\lambda}$. Then Lemma \ref{Lemma:CritIndes} yields a club $C$ in $\kappa$ and a $\kappa$-list $\vec{e}=\seq{e_\alpha}{\alpha<\kappa}$ with the property that whenever $\theta$ is a sufficiently large cardinal such that there is a small embedding $\map{j}{M}{\HH{\theta}}$ for $\kappa$ witnessing the subtlety of $\kappa$ with respect to $\vec{e}$ and $C$ as in statement (ii) of Lemma \ref{Lemma:CritIndes}, then $\crit{j}$ is an inaccessible cardinal.  
 
 Let $A$ denote the set of all $a\in\Poti{\lambda}{\kappa}$ with the property that there is a cardinal $\vartheta_a>\lambda$ and an elementary submodel $X_a$ of $\HH{\vartheta_a}$ such that $f\in X_a$, $\alpha_a=X_a\cap\kappa\in C$ is inaccessible and $\Poti{X_a\cap\lambda}{\alpha_a}\nsubseteq X_a$. Given $a\in A$, pick $x_a\in\Poti{X_a\cap\lambda}{\alpha_a}\setminus X_a$. Next, let $\vec{d}=\seq{d_a}{a\in\Poti{\lambda}{\kappa}}$ denote the unique $\Poti{\lambda}{\kappa}$-list such that $d_a=x_a$ for all $a\in A$, $d_a=e_{a\cap\kappa}$ for all $a\in\Poti{\lambda}{\kappa}\setminus A$ with $a\cap\kappa\in C$, and $d_a=\emptyset$ otherwise. 
  
 Now, let $\theta$ be a sufficiently large cardinal such that there is a small embedding $\map{j}{M}{\HH{\theta}}$ and $\delta\in M\cap\kappa$ witnessing the $\lambda$-ineffability of $\kappa$ with respect to $\vec{d}$, as in statement (ii) of Lemma \ref{lemma:LambdaIneffableSmallChar}, such that $f$, $\vec{e}$ and $C$ are contained in $\ran{j}$. Assume for a contradiction that either $\crit{j}$ is not inaccessible or $\Poti{\delta}{\crit{j}}\nsubseteq M$. 
 
Next, assume also that $j[\delta]\notin A$. Since $j[\delta]\cap\kappa=\crit{j}\in C$, $j^{{-}1}[d_{j[\delta]}]\in M$ implies that $e_{\crit{j}}\in M$. In this situation, the combination of Lemma \ref{Lemma:CritIndes} and Lemma \ref{lemma:IneffableWitnessSubtle} yields that $\crit{j}=j[M]\cap\kappa$ is inaccessible. Since our assumptions imply that $\Poti{j[M]\cap\lambda}{\crit{j}}\nsubseteq j[M]$, we can conclude that $j[M]$ witnesses that $j[\delta]\in A$, a contradiction. 
 
Hence $j[\delta]\in A$. Since we know that $\crit{j}=\alpha_{j[\delta]}$ and $j^{{-}1}[d_{j[\delta]}]\in M$, we know that $x_{j[\delta]}\in j[M]$. But this allows us to conclude that $f(x_{j[\delta]})\in \lambda\cap j[M]\subseteq X_{j[\delta]}$ and hence $x_{j[\delta]}\in X_{j[\delta]}$, a contradiction. 
\end{proof}


The small embedding characterization of $\lambda$-ineffable cardinals suggests a natural strengthening of $\lambda$-ineffability that arises from a modification of the quantifiers that appear in Statement (ii) of Lemma \ref{lemma:LambdaIneffableSmallChar}.

\begin{definition}\label{superineffable}
 Given cardinals $\kappa\leq\lambda$, the cardinal $\kappa$ is \emph{$\lambda$-superineffable} if for all sufficiently large cardinals $\theta$, there is a small embedding $\map{j}{M}{\HH{\theta}}$ for $\kappa$ and $\delta\in M\cap\kappa$ with the property that $j(\delta)=\lambda$ and $j^{{-}1}[d_{j[\delta]}]\in M$ holds for every $\Poti{\lambda}{\kappa}$-list $\vec{d}=\seq{d_a}{a\in\Poti{\lambda}{\kappa}}$ with $\vec{d}\in\ran{j}$.  
\end{definition}

\begin{proposition}\label{proposition:SuperineffableIneffable}
 Assume that $\kappa$ is $\lambda$-superineffable. Then amongst the embeddings witnessing the $\lambda$-superineffability of $\kappa$ are embeddings witnessing that $\kappa$ is $\lambda$-ineffable and $\lambda_0$-superineffable for all cardinals $\kappa\leq\lambda_0<\lambda$.  
\end{proposition}

\begin{proof}
 The first statement follows directly from Lemma \ref{lemma:addxtorange}, because the correctness  property used in Definition \ref{superineffable} is easily seen to be downwards-absolute, and therefore Lemma \ref{lemma:addxtorange} allows us to capture any $\Poti{\lambda}{\kappa}$-list in the range of such a small embedding. The second statement follows from a small modification of Proposition \ref{proposition:DownwardsLambdaIneffable}, showing that every small embedding $\map{j}{M}{\HH{\theta}}$ for $\kappa$ witnessing the $\lambda$-superineffability of  $\kappa$ also witnesses the  $\lambda_0$-superineffability of  $\kappa$ for every cardinal $\kappa\leq\lambda_0<\lambda$ in $\ran{j}$. 
\end{proof}

Corollary \ref{corollary:SupercompactMagidorStyle} directly  provides us with examples of $\lambda$-superineffable cardinals. We will later improve the following implication by showing that small embeddings witnessing $\lambda$-supercompactness also witness $\lambda$-superineffability (see Lemma \ref{lemma:measurablesmallembeddingconsequences} and Statement (iii) of Lemma \ref{lemma:ImplicationsSECLambdaSupercompact}), hence in particular $\lambda$-supercompact cardinals are $\lambda$-superineffable.

\begin{corollary}\label{corollary:SupercompactSuperineffable}
  Let $\kappa$ be a supercompact cardinal. If $\theta$ is a sufficiently large cardinal such that there is a small embedding $\map{j}{M}{\HH{\theta}}$ for $\kappa$ witnessing the supercompactness of $\kappa$, as in statement (ii) of Corollary \ref{corollary:SupercompactMagidorStyle}, then $j$ witnesses the $\lambda$-superineffability of $\kappa$ for every cardinal $\lambda\geq\kappa$ with $\lambda\in\ran{j}$. \qed
\end{corollary}

The results of \cite{MR0327518} show that a cardinal $\kappa$ is supercompact if and only if it is $\lambda$-ineffable for every cardinal $\lambda\geq\kappa$. By combining this result with Lemma \ref{lemma:LambdaIneffableSmallChar}, Proposition \ref{proposition:SuperineffableIneffable} and Corollary \ref{corollary:SupercompactSuperineffable}, we obtain the following alternative small embedding characterizations of supercompactness.

\begin{corollary} 
  The following statements are equivalent for every cardinal $\kappa$: 
 \begin{enumerate}[leftmargin=0.7cm]
  \item $\kappa$ is supercompact. 
   \item For all cardinals $\lambda\ge\kappa$, all sufficiently large cardinals $\theta$ and every $\Poti{\lambda}{\kappa}$-list $\vec d=\seq{d_a}{a\in\Poti{\lambda}{\kappa}}$, there is a small embedding $\map{j}{M}{\HH{\theta}}$ for $\kappa$ such that $\betrag{M}<\kappa$, $\lambda,\vec{d}\in\ran{j}$, and $j^{{-}1}[d_{\ran{j}\cap\lambda}]\in M$. 
  \item For all sufficiently large cardinals $\theta$, there is a small embedding $\map{j}{M}{\HH{\theta}}$ for $\kappa$ such that $\betrag{M}<\kappa$ and $j^{{-}1}[d_{\ran{j}\cap\lambda}]\in M$ for every cardinal $\lambda\geq\kappa$ and every $\Poti{\lambda}{\kappa}$-list $\vec{d}=\seq{d_a}{a\in\Poti{\lambda}{\kappa}}$ with $\lambda,\vec{d}\in\ran{j}$. \qed
 \end{enumerate}
\end{corollary}

A similar, but perhaps seemingly less natural strengthening can also be obtained for the notion of subtlety.

\begin{definition}\label{supersubtle}
  A cardinal $\kappa$ is \emph{supersubtle} if for all sufficiently large cardinals $\theta$ and for every club $C$ in $\kappa$, there is a small embedding $\map{j}{M}{\HH{\theta}}$ for $\kappa$ with $C\in\ran j$ and the property that whenever $\vec{d}=\seq{d_\alpha}{\alpha<\kappa}$ is a $\kappa$-list in $\ran{j}$, then there is an $\alpha\in C\cap\crit{j}$ with $d_\alpha=d_\crit{j}\cap\alpha$.  
\end{definition}

We start by observing what should be of no suprise, namely that $\lambda$-superineffable cardinals are supersubtle.

\begin{lemma}
 Let $\kappa$ be a $\kappa$-superineffable cardinal and let $C$ be a club in $\kappa$. If $\theta$ is a sufficiently large cardinal such that there is a small embedding $\map{j}{M}{\HH{\theta}}$ for $\kappa$  witnessing the $\kappa$-superineffability of $\kappa$, then $C\in\ran{j}$ implies that $j$ also witnesses the supersubtlety of $\kappa$ with respect to $C$. 
\end{lemma}

\begin{proof}
 Let $\seq{d_\alpha}{\alpha<\kappa}$ be a $\kappa$-list in $\ran{j}$. Then the $\kappa$-superineffability of $\kappa$ implies that $d_{\crit{j}}\in M$ and therefore $d_{\crit{j}}=j(d_{\crit{j}})\cap\crit{j}$. Since $C$ is an element of $\ran{j}$, we have $\crit{j}\in C$ and elementarity implies that there is an $\alpha\in C\cap\crit{j}$ with $d_\alpha=j(d_{\crit{j}})\cap\alpha=d_{\crit{j}}\cap\alpha$.  
\end{proof}

We now show that supersubtle cardinals are downwards absolute to $\LL$. The proof of this statement relies on a classical argument of Kenneth Kunen, which is commonly referred to as the \emph{ancient Kunen lemma}.



\begin{lemma}\label{supersubtleinl}
  If $\kappa$ is supersubtle, then $\kappa$ is supersubtle in $\LL$.
\end{lemma}

\begin{proof}
 Assume that $\lambda$ is an ordinal such that for every cardinal $\theta>\lambda$ and every club $C$ in $\kappa$, there is a small embedding $\map{j}{M}{\HH{\theta}}$ for $\kappa$ witnessing the supersubtlety of $\kappa$ with respect to $C$. Fix an $\LL$-cardinal $\theta>\lambda$, a cardinal $\vartheta>\theta$ and a constructible club $C$ in $\kappa$. By our assumptions and Lemma \ref{lemma:addxtorange}, we can find a small embedding $\map{j}{M}{\HH{\vartheta}}$ for $\kappa$ witnessing the supersubtlety of $\kappa$ with respect to $C$ and $\delta\in M$ with $j(\delta)=\theta$. Since $\theta$ is a cardinal greater than $\kappa$ in $\LL$, elementarity implies that $j^{{-}1}(C)$ is an element of $\LL_\delta$. Let $X\in\LL_\delta$ denote the Skolem hull of $\{j^{{-}1}(C)\}\cup(\crit{j}+1)$ in $\LL_\delta$. Then $X$ has cardinality $\crit{j}$ in $\LL_\delta$ and we can fix a bijection $\map{f}{\crit{j}}{X}$ in $\LL_\delta$. Then $$j\upharpoonright X ~ = ~ \Set{\langle f(\alpha), ~ j(f)(\alpha)\rangle}{\alpha<\crit{j}}.$$ Since $\map{j\restriction\LL_\delta}{\LL_\delta}{\LL_\theta}$ is an elementary embedding and $f,j(f),X,j(X)\in\LL_\theta$, we can conclude that $j\restriction X$ is an element of $\LL$. Let $\map{\pi}{X}{\LL_\varepsilon}$ denote the transitive collapse of $X$. Then $\map{j\circ\pi^{{-}1}}{\LL_\varepsilon}{\LL_\theta}$ is a constructible small embedding for $\kappa$ with $C\in\ran{j\circ\pi^{{-}1}}$. 
 
 Let $\vec{d}=\seq{d_\alpha}{\alpha<\kappa}$ be a $\kappa$-list in $\ran{j\circ\pi^{{-}1}}$. Since $\vec{d}\in\ran{j}$, there is an $\alpha\in C\cap\crit{j}$ with $d_\alpha=d_{\crit{j}}\cap\alpha$. Since we have $\crit{j}=\crit{j\circ\pi^{{-}1}}$, this shows that $j\circ\pi^{{-}1}$ witnesses the supersubtlety of $\kappa$ with respect to $C$ in $\LL$. 
\end{proof}

We do not know whether $\kappa$-superineffable cardinals are downwards absolute to $\LL$ (see Question \ref{superineffablequestion}).


\section{Filter-based large cardinals}\label{section:filterbased}

 Next, we show that large cardinal notions defined through the existence of certain normal filters can also be characterized through the existence of small embeddings. We start by considering measurable cardinals. The proof of their small embedding characterization is almost the same as the proof of the small embedding characterization of $\lambda$-supercompact cardinals in Lemma \ref{lemma:SupercompactLevelSmallChar} below, however we would like to provide a full proof here as well, for the convenience of the reader.

\begin{lemma}\label{lemma:MeasureableSmallChar}
 The following statements are equivalent for every cardinal $\kappa$: 
 \begin{enumerate}[leftmargin=0.7cm]
  \item $\kappa$ is measurable. 

  \item For all sufficiently large cardinals $\theta$,  there is a small embedding $\map{j}{M}{\HH{\theta}}$ for $\kappa$ with $$\Set{A\in\POT{\crit{j}}^M}{\crit{j}\in j(A)} ~ \in ~ M.$$
 \end{enumerate}
\end{lemma} 

\begin{proof}
  Assume that $U$ is a normal ultrafilter on $\kappa$ witnessing the measurability of $\kappa$ and let $\map{j_\calU}{\VV}{\Ult{\VV}{U}}$ denote the corresponding ultrapower embedding. Pick a cardinal $\theta>2^\kappa$ and an elementary submodel $X$ of $\HH{\theta}$ of cardinality $\kappa$ containing $\{U\}\cup(\kappa+1)$. Let $\map{\pi}{X}{N}$ denote the corresponding transitive collapse. Then the map $\map{k=j_U\circ\pi^{{-}1}}{N}{\HH{j_U(\theta)}^{\Ult{\VV}{U}}}$ is a non-trivial elementary embedding with $\crit{k}=\kappa$ and $k(\crit{k})=j_U(\kappa)$. Since $\Ult{\VV}{U}$ is closed under $\kappa$-sequences in $\VV$ and $N\in\HH{\kappa^+}\subseteq\Ult{\VV}{U}$, we have $k,N\in\Ult{\VV}{U}$ and the map $\map{k}{N}{\HH{j_U(\theta)}^{\Ult{\VV}{U}}}$ is a small embedding for $j_U(\kappa)$ in $\Ult{\VV}{U}$. Given $A\in\POT{\kappa}^N$, we have $\pi^{{-}1}(A)=A=\pi(A)$ and hence $$\kappa\in k(A) ~ \Longleftrightarrow ~  \kappa\in j_U(A) ~ \Longleftrightarrow ~ A\in U ~ \Longleftrightarrow ~ A
  \in\pi(U).$$ This allows us to conclude that $$\pi(U) ~ = ~ \Set{A\in \POT{\kappa}^N}{\kappa\in k(A)} ~ \in ~ N.$$ Using elementarity, we can find a small embedding $\map{j}{M}{\HH{\theta}}$ for $\kappa$ in $\VV$ with the property stated in (ii).

 Now, assume that (ii) holds. Pick a sufficiently large cardinal $\theta$ and a small embedding $\map{j}{M}{\HH{\theta}}$ for $\kappa$ such that $\theta>2^\kappa$ and the set $U$ of all $A\in\POT{\crit{j}}^M$ with $\crit{j}\in j(A)$ is contained in $M$. Then $U$ is a normal ultrafilter on $\crit{j}$ in $M$. Since $\theta>2^\kappa$, this shows that $j(U)$ is a normal ultrafilter on $\kappa$ that witnesses the measurability of $\kappa$.  
\end{proof}

The following observation connects the above result  with the small embedding characterizations of smaller large cardinal notions, by showing that witnessing small embeddings for measurability are also witnessing embeddings for all large cardinal notions considered so far that are direct consequences of measurability.

\begin{lemma}\label{lemma:measurablesmallembeddingconsequences}
 Let $\kappa$ be a measurable cardinal, let $\theta>2^\kappa$ be a sufficiently large cardinal such that there is a small embedding $\map{j}{M}{\HH{\theta}}$ for $\kappa$ witnessing the measurability of $\kappa$, as in statement (ii) of Lemma \ref{lemma:MeasureableSmallChar}. 
 \begin{enumerate}
  \item The embedding $j$ witnesses that $\kappa$ is a stationary limit of Ramsey cardinals, as in statement (ii) of Lemma \ref{lemma:CharStatLimit}. 
  
  \item If $x\in\VV_{\kappa+1}\cap\ran{j}$, then  $j$ witnesses the $\Pi^2_1$-indescribability of $\kappa$ with respect to $x$, as in statement (iii) of Lemma \ref{lemma:SmallEmbCharIndescribable}. 
  
  \item The embedding $j$ witnesses that the $\kappa$-superineffability of $\kappa$. 
 \end{enumerate}
\end{lemma}

\begin{proof}
   Let $U$ denote the set of all $A\in\POT{\crit{j}}^M$ with $\crit{j}\in j(A)$. Then $U$ is an element of $M$ and $j(U)$ is a normal ultrafilter on $\kappa$. 
   
   (i) Since $\kappa$ is a Ramsey cardinal in the ultrapower $\Ult{\VV}{j(U)}$, it follows that the set of all Ramsey cardinals less than $\kappa$ is an element of $j(U)$ and this implies that $\crit{j}$ is a Ramsey cardinal. 
   
   (ii) Let $\varphi(v)$ be a $\Pi^2_1$-formula with $(\VV_{\crit{j}}\models\varphi(j^{{-}1}(x)))^M$. Then elementarity implies $\VV_\kappa\models\varphi(x)$ and we can apply {\cite[Proposition 6.5]{MR1994835}} to conclude that the set $\Set{\alpha<\kappa}{\VV_\alpha\models\varphi(x\cap\VV_\alpha)}$ is an element of $j(U)$. Since Lemma \ref{lemma:ModelsContainInitialSegment} implies that $j^{{-}1}(x)=x\cap\VV_{\crit{j}}$, this shows that $\VV_{\crit{j}}\models\varphi(j^{{-}1}(x))$. 
   
 (iii) Let $\vec{d}=\seq{d_\alpha}{\alpha<\kappa}$ be a $\kappa$-list in $\ran{j}$ and let $d$ denote the set of all $\alpha<\crit{j}$ with the property that the set $D_\alpha=\Set{\beta<\crit{j}}{\alpha\in d_\beta}$ is contained in $U$. By our assumptions, the set $d$ is an element of $M$. Then we have $$\alpha\in d ~ \Longleftrightarrow ~ D_\alpha\in U ~ \Longleftrightarrow ~ \crit{j}\in j(D_\alpha) ~ \Longleftrightarrow ~ \alpha\in d_{\crit{j}}$$ for all $\alpha<\crit{j}$ and this shows that $d=d_{\crit{j}}\in M$. 
\end{proof}

Lemma \ref{lemma:MeasureableSmallChar} directly generalizes to a small embedding characterization of certain degrees of supercompactness.

\begin{lemma}\label{lemma:SupercompactLevelSmallChar}
 The following statements are equivalent for all cardinals $\kappa\leq\lambda$:  
 \begin{enumerate}[leftmargin=0.7cm]
  \item $\kappa$ is $\lambda$-supercompact. 

   \item For all sufficiently large cardinals $\theta$,  there is a small embedding $\map{j}{M}{\HH{\theta}}$ for $\kappa$ and $\delta\in M\cap\kappa$ such that $j(\delta)=\lambda$ and $$\Set{A\in\POT{\Poti{\delta}{\crit{j}}}^M}{j[\delta]\in j(A)} ~ \in ~  M.$$
 \end{enumerate} 
\end{lemma}

\begin{proof}
   Assume that there is a normal ultrafilter $U$ on $\Poti{\lambda}{\kappa}$ witnessing the $\lambda$-supercompactness of $\kappa$. Let $\map{j_U}{\VV}{\Ult{\VV}{U}}$ denote the corresponding ultrapower embedding. Then $\lambda<j_U(\kappa)$. Fix a cardinal $\theta$ with $U\in\HH{\theta}$ and an elementary submodel $X$ of $\HH{\theta}$ of cardinality $\lambda$  with $\{U\}\cup(\lambda+1)\subseteq X$.  Let $\map{\pi}{X}{N}$ denote the corresponding transitive collapse. Then the closure of $\Ult{\VV}{U}$ under $\lambda$-sequences in $\VV$ implies that the map $\map{k=j_U\circ\pi^{{-}1}}{N}{\HH{j_U(\theta)}^{\Ult{\VV}{U}}}$ is an element of $\Ult{\VV}{\calU}$, and this map is a small embedding for $j_U(\kappa)$ with $\crit{k}=\kappa$ and  $k(\lambda)=j_U(\lambda)$ in $\Ult{\VV}{U}$. Then $k[\lambda]=j_U[\lambda]$ and therefore we have $$k[\lambda]\in k(A) ~ \Longleftrightarrow ~  j_U[\lambda]\in j_U(\pi^{{-}1}(A))
 ~ \Longleftrightarrow ~ \pi^{{-}1}(A)\in U ~ \Longleftrightarrow ~ A\in\pi(U)$$ for all $A\in \POT{\Poti{\lambda}{\kappa}}^N$. These computations show that $$\pi(U) ~ = ~ \Set{A\in\POT{\Poti{\lambda}{\kappa}}^N}{k[\lambda]\in k(A)} ~ \in ~ N.$$ In this situation, we can use elementarity between $\VV$ and $\Ult{\VV}{U}$ to find a small embedding $\map{j}{M}{\HH{\theta}}$ for $\kappa$ and $\delta\in M$ such that $\delta<\kappa$, $j(\delta)=\lambda$ and $\Set{A\in\POT{\Poti{\delta}{\crit{j}}}^{M}}{j[\delta]\in j(A)}\in M$.

 Now, assume that (ii) holds. Fix a cardinal $\theta$ such that $\POT{\Poti{\lambda}{\kappa}}\in\HH{\theta}$ and such that there is a small embedding $\map{j}{M}{\HH{\theta}}$ for $\kappa$ and $\delta\in M\cap\kappa$ as in (ii). Then the set $U$ of all $A\in\POT{\POTI{\delta}{\crit{j}}}^M$ with $j[\delta]\in j(A)$ is an element of $M$ and the assumption $\delta<\kappa$ implies that this set is a normal ultrafilter on $\Poti{\delta}{\crit{j}}$ in $M$. Since $\POT{\Poti{\lambda}{\kappa}}\in\HH{\theta}$, we can conclude that $j(U)$ is a normal filter on $\Poti{\lambda}{\kappa}$ that witnesses the $\lambda$-supercompactness of $\kappa$.  
\end{proof}


\begin{lemma}\label{lemma:ImplicationsSECLambdaSupercompact}
 Let $\kappa$ be a $\lambda$-supercompact cardinal, let  $\theta$ be a sufficiently large cardinal such that there is a small embedding $\map{j}{M}{\HH{\theta}}$ for $\kappa$ witnessing the $\lambda$-supercompactness of $\kappa$, as in statement (ii) of Lemma \ref{lemma:SupercompactLevelSmallChar}. 
 \begin{enumerate}[leftmargin=0.7cm]
  \item The embedding $j$ witnesses the measurability of $\kappa$, as in Lemma \ref{lemma:MeasureableSmallChar}. 

  \item If $\kappa\leq\lambda_0<\lambda$ is a cardinal in \ran{j}, then the embedding $j$ witnesses the $\lambda_0$-supercompactness of $\kappa$, as in statement (ii) of Lemma \ref{lemma:SupercompactLevelSmallChar}. 

  \item The embedding $j$ witnesses the $\lambda$-superineffabilty of $\kappa$. 
 \end{enumerate}
\end{lemma}

\begin{proof}
 Pick $\delta\in M\cap\kappa$ with $j(\delta)=\lambda$ and let $U\in M$ denote the set of all $A\in\POT{\Poti{\delta}{\crit{j}}}^M$ with $j[\delta]\in j(A)$. 

 (i) Define $F$ to be the set of all $x\in\POT{\crit{j}}^M$ with the property that the set $F_x=\Set{a\in\Poti{\delta}{\crit{j}}^M}{\otp{a\cap\crit{j}}\in x}$ is an element of $U$. Then our assumptions imply that $F$ is an element of $M$ and we have $$x\in F ~ \Longleftrightarrow ~ F_x\in U ~ \Longleftrightarrow ~ j[\delta]\in j(F_x) ~ \Longleftrightarrow ~ \crit{j}=\otp{j[\delta]\cap\kappa}\in j(x)$$ for all $x\in\POT{\crit{j}}^M$. 

 (ii) Pick $\delta_0\in M$ with $j(\delta_0)=\lambda_0$. Let $F$ denote the set of all $A\in\POT{\Poti{\delta_0}{\crit{j}}}^M$ with the property that the set $\Set{a\in\Poti{\delta}{\crit{j}}^M}{a\cap\delta_0\in A}$ is an element of $U$. Then $F$ is an element of $M$ and it is equal to the set of all $A\in\POT{\Poti{\delta_0}{\crit{j}}}^M$ with $j[\delta_0]\in j(A)$.  

 (iii) Pick a $\Poti{\lambda}{\kappa}$-list $\vec{d}=\seq{d_a}{a\in\Poti{\lambda}{\kappa}}$ in $\ran{j}$ and a $\Poti{\delta}{\crit{j}}$-list $\vec{e}=\seq{e_a}{a\in\Poti{\delta}{\crit{j}}}$ in $M$ with $j(\vec{e})=\vec{d}$. Let $D$ denote the set of all $\gamma<\delta$ with the property that the set $D_\gamma=\Set{a\in\Poti{\delta}{\crit{j}}^M}{\gamma\in e_a}$ is contained in $U$. Then $D$ is an element of $M$ and we have $$\gamma\in D ~ \Longleftrightarrow ~ D_\gamma\in U ~ \Longleftrightarrow ~ j[\delta]\in j(D_\gamma) ~ = ~ j(\gamma)\in d_{j[\delta]}$$ for all $\gamma<\delta$. This shows that $D=j^{{-}1}[d_{j[\delta]}]\in M$. 
\end{proof}

The next proposition shows that the domain models of small embeddings witnessing $\lambda$-supercompactness possess certain closure properties. These closure properties will allow us to connect the characterization of supercompactness provided by Lemma \ref{lemma:SupercompactLevelSmallChar} with Magidor's characterization in Corollary \ref{corollary:SupercompactMagidorStyle}.

\begin{proposition}\label{proposition:ClosureLambdaSupercompactSEC}
 Let $\kappa$ be a $\lambda$-supercompact cardinal and let  $\map{j}{M}{\HH{\theta}}$ be a small embedding for $\kappa$ witnessing the $\lambda$-supercompactness of $\kappa$, as in statement (ii) of Lemma \ref{lemma:SupercompactLevelSmallChar}.  If $\delta\in M\cap\kappa$ with $j(\delta)=\lambda$ and $x\in\POT{\crit{j}}^M$, then $j(x)\cap\delta\in M$. Moreover, if $\lambda$ is a strong limit cardinal, then $\delta$ is a strong limit cardinal and $\HH{\delta}\in M$. 
\end{proposition}

\begin{proof}
 Fix some $x\in\POT{\crit{j}}^M$. Given $\gamma<\delta$, set $$A_\gamma=\Set{a\in\Poti{\delta}{\crit{j}}^M}{\gamma\in a, ~ \otp{a\cap\gamma}\in x}.$$ Then $$j[\delta]\in j(A_\gamma) ~ \Longleftrightarrow ~ \otp{ j[\delta]\cap j(\gamma)}\in j(x) ~ \Longleftrightarrow ~ \gamma\in j(x)$$ for all $\gamma<\delta$. By our assumptions, these equivalences imply that the subset $j(x)\cap\delta$ is definable in $M$. 
 
 Now, assume that $\lambda$ is a strong limit cardinal. Fix a sequence $s=\seq{s_\alpha}{\alpha<\crit{j}}$ in $M$ such that $\map{s_\alpha}{(2^{\betrag{\alpha}})^M}{\POT{\alpha}^M}$ is a bijection for every $\alpha<\crit{j}$. Define $$x ~ = ~ \Set{\goedel{\alpha}{\goedel{\beta}{\gamma}}}{\alpha<\crit{j}, ~ \beta<2^{\betrag{\alpha}}, \gamma\in s_\alpha(\beta)} ~ \in ~ \POT{\crit j}^M.$$ Elementarity implies that $\delta$ is a strong limit cardinal in $M$, and the above computations show that $j(x)\cap\delta$ is an element of $M$.  Assume for a contradiction that $\delta$ is not a strong limit cardinal. Pick a cardinal $\nu<\delta$ with $2^\nu\geq\delta$. Then the injection $\map{j(s)_\nu\restriction\delta}{\delta}{\POT{\nu}}$ can be defined from $j(x)\cap\delta$, and therefore this function is contained in $M$, a contradiction. Since the above computations show that the sequence $\seq{j(s)_\alpha}{\alpha<\delta}$ can be defined from the subset $j(x)\cap\delta$ of $\delta$ and this subset is contained in $M$, it follows that $\HH{\delta}$ is an element of $M$. 
\end{proof}

\begin{corollary}
 Let $\kappa$ be a supercompact cardinal. 
 \begin{enumerate}[leftmargin=0.7cm]
  \item Let $\lambda\geq\kappa$ be a cardinal and let $\theta>2^{(\lambda^{{<}\kappa})}$ be a sufficiently large cardinal such that there is a small embedding $\map{j}{M}{\HH{\theta}}$ for $\kappa$ witnessing the supercompactness of $\kappa$, as in statement (ii) of Corollary \ref{corollary:SupercompactMagidorStyle}. If $\lambda\in\ran{j}$, then $j$ witnesses the $\lambda$-supercompactness of $\kappa$, as in statement (ii) of Lemma \ref{lemma:SupercompactLevelSmallChar}. 
 
  \item Let $\theta>\kappa$ be a cardinal, let $\lambda\geq\theta$ be a strong limit cardinal and let $\vartheta$ be sufficiently large such that there exists a small embedding $\map{j}{M}{\HH{\vartheta}}$ for $\kappa$ witnessing the $\lambda$-supercompactness of $\kappa$, as in statement (ii) of Lemma \ref{lemma:SupercompactLevelSmallChar}. If there is a $\delta\in M$ with $j(\delta)=\theta$, then $\map{j\restriction\HH{\delta}^M}{\HH{\delta}}{\HH{\theta}}$ witnesses the supercompactness of $\kappa$, as in statement (ii) of Corollary \ref{corollary:SupercompactMagidorStyle}.  \qed
 \end{enumerate}
\end{corollary}

In the remainder of this section, we turn our attention to \emph{huge} cardinals. Remember that, given $0<n<\omega$, an uncountable cardinal $\kappa$ is \emph{$n$-huge} if there is a sequence $\kappa=\lambda_0<\lambda_1<\ldots<\lambda_n$ of cardinals and a $\kappa$-complete normal ultrafilter $U$ on $\POT{\lambda_n}$ with $\Set{a\in\POT{\lambda_n}}{\otp{a\cap\lambda_{i+1}}=\lambda_i}\in U$ for all $i<n$.  A cardinal  is \emph{huge} if it is $1$-huge. Note that, if $\lambda_0<\lambda_1<\ldots<\lambda_n$ and $U$ witness the $n$-hugeness of $\kappa$ and $\map{j_U}{\VV}{\Ult{\VV}{U}}$  is the induced ultrapower embedding, then $\crit{j_U}=\kappa$, $j_U(\lambda_i)=\lambda_{i+1}$ for all $i<n$, $U=\Set{A\in\POT{\POT{\lambda_n}}}{j_U[\lambda_n]\in j_U(A)}$ and $\Ult{\VV}{U}$ is closed under $\lambda_n$-sequences.  In particular, each $\lambda_i$ is measurable. 
Moreover, since $U$ concentrates on the subset  $[\lambda_n]^{\lambda_{n-1}}$ of all subsets of $\lambda_n$ of order-type $\lambda_{n-1}$, we may as well identify $U$ with an ultrafilter on this set of size   $\lambda_n$.

\begin{lemma}\label{lemma:CharacterizeNHuge}
  Given $0<n<\omega$, the following statements are equivalent for all cardinals $\kappa$: 
 \begin{enumerate}[leftmargin=0.7cm]
  \item $\kappa$ is $n$-huge. 

   \item For all sufficiently large cardinals $\theta$, there is a small embedding $\map{j}{M}{\HH{\theta}}$ for $\kappa$ such that $j^i(\crit{j})\in M$ for all $i\leq n$ and  $$\Set{A\in\POT{\POT{j^n(\crit{j})}}^M}{j[j^n(\crit{j})]\in j(A)} ~ \in ~ M.$$ 
 \end{enumerate} 
\end{lemma}

\begin{proof}
  First, assume that $\lambda_0<\lambda_1<\ldots<\lambda_n$ and $U$ witness the $n$-hugeness of $\kappa$ and let $\map{j_U}{\VV}{\Ult{\VV}{U}}$ denote the corresponding ultrapower embedding. Pick a cardinal $\theta$ with $U\in\HH{\theta}$ and an elementary submodel $X$ of $\HH{\theta}$ of cardinality $\lambda_n$ with $\HH{\lambda_n}\cup\{U\}\subseteq X$. Let $\map{\pi}{X}{N}$ denote the corresponding transitive collapse. Then $k=\map{j_U\circ\pi}{N}{\HH{j_U(\theta)}^{\Ult{\VV}{U}}}$ is a non-trivial elementary embedding and $\pi^{{-}1}\restriction\HH{\lambda_n}=\id_{\HH{\lambda_n}}$ implies that $\crit{k}=\kappa$, $k[\lambda_n]=j_U[\lambda_n]$ and $k(\lambda_i)=\lambda_{i+1}$ for all $i<n$. Since $N\in\HH{\lambda_n^+}\subseteq\Ult{\VV}{U}$, the closure of $\Ult{\VV}{U}$ under $\lambda$-sequences implies that $k$ is an element of $\Ult{\VV}{U}$ and the above computations show that $k$ is a small embedding for $j_U(\kappa)$ in $\Ult{\VV}{U}$. Then $$k[\lambda_n]\in k(A) ~ \Longleftrightarrow ~  j_U[\lambda_n]\in j_U(\pi^{{-}1}(A)) ~ \Longleftrightarrow ~ \pi^{{-}1}(A)\in U ~ \Longleftrightarrow ~ A\in\pi(U).$$
for all $A\in\POT{\POT{\lambda_n}}^N$. This shows that $$\pi(U) ~ = ~ \Set{A\in\POT{\POT{k^n(\crit{k})}}^N}{k[k^n(\crit{k})]\in k(A)} ~ \in ~ N$$ and, by elementarity, we find a small embedding for $\kappa$ as in (ii).  

  Now, assume that (ii) holds. Fix a sufficiently large cardinal $\theta$ and a small embedding $\map{j}{M}{\HH{\theta}}$ for $\kappa$ as in (ii). Let  $U$ denote the set of all $A$ in $\POT{\POT{j^n(\crit{j})}}^M$ with $j[j^n(\crit{j})]\in j(A)$ and set $\lambda_i=j^i(\crit{j})$ for all $i\leq n$.  Then $U$ is an element of $M$ and our assumptions imply that $U$ is a $\crit{j}$-complete, normal ultrafilter on $\POT{\lambda_n}$ with $\Set{a\in\POT{\lambda_n}^M}{\otp{a\cap\lambda_{i+1}}=\lambda_i}\in U$ for all $i<n$. Hence $\lambda_0<\lambda_1<\ldots<\lambda_n$ and $U$ witness that $\crit{j}$ is an $n$-huge cardinal in $M$. This allows us to conclude that $\lambda_1<\ldots<\lambda_n<j(\lambda_n)$ and $j(U)$ witness that $\crit{j}$ is an $n$-huge cardinal in $\HH{\theta}$. Then $j(U)$ concentrates on the subset $[j(\lambda_n)]^{\lambda_n}$ of $\POT{j(\lambda_n)}$. Since $j(\lambda_n)$ is inaccessible and therefore $\POT{[j(\lambda_n)]^{\lambda_n}}$ is contained in $\HH{\theta}$, we can conclude that $\kappa$ is $n$-huge. 
\end{proof}

The next lemma shows that the domain models of small embeddings witnessing $n$-hugeness also possess certain closure properties. These closure properties will directly imply that these  embeddings also witness weaker large cardinal properties.

\begin{lemma}\label{hugeclosure}
 Let $0<n<\omega$, let $\kappa$ be an $n$-huge cardinal and let $\map{j}{M}{\HH{\theta}}$ be a small embedding for $\kappa$ witnessing the $n$-hugeness of $\kappa$, as in statement (ii) of  Lemma \ref{lemma:CharacterizeNHuge}. Then $\POT{j^n(\crit{j})}\cap\ran{j}$ is contained in $M$. In particular, $\HH{j^n(\crit{j})}$ is an element of $M$. 
\end{lemma}

\begin{proof}
 Fix $A\in\POT{j^{n-1}(\crit{j})}^M$. Given $\gamma<j^n(\crit{j})$, define $$A_\gamma ~ = ~ \Set{a\in\POT{j^n(\crit{j})}^M}{\gamma\in a, ~ \otp{a\cap\gamma}\in A}.$$ 
For each $\gamma<j^n(\crit{j})$, we then have 
\begin{equation*}
 \begin{split}
  A_\gamma\in U ~ & \Longleftrightarrow ~ j[j^n(\crit{j})]\in j(A_\gamma) ~ \Longleftrightarrow ~ \otp{j(\gamma)\cap j[j^n(\crit{j})]}\in j(A) \\
  & \Longleftrightarrow ~ \otp{j[\gamma]}\in j(A) ~ \Longleftrightarrow ~ \gamma \in j(A).
 \end{split}
\end{equation*}
This shows that $j(A)$ is equal to the set $\Set{\gamma<j^n(\crit{j})}{A_\gamma\in U}$. Since the sequence $\seq{A_\gamma}{\gamma<j^n(\crit{j})}$ is an element of $M$, this shows that $j(A)\in M$. 

The final statement of the lemma follows from the fact that elementarity implies that there is a subset of $j^n(\crit{j})$ in $\ran{j}$ that codes all elements of $\HH{j^n(\crit{j})}$. 
\end{proof}

\begin{corollary}
  Let $0<n<\omega$, let $\kappa$ be an $n$-huge cardinal and let $\theta$ be a sufficiently large cardinal such that there is a small embedding $\map{j}{M}{\HH{\theta}}$ for $\kappa$ witnessing the $n$-hugeness of $\kappa$, as in statement (ii) of  Lemma \ref{lemma:CharacterizeNHuge}. 
  \begin{enumerate}[leftmargin=0.7cm]
    \item If $0<m<n$, then $j$ also witnesses the $m$-hugeness of $\kappa$, as in statement (ii) of  Lemma \ref{lemma:CharacterizeNHuge}. 
    
    \item If $\kappa\leq \lambda<j(\kappa)$ with $\lambda\in\ran{j}$, then $j$ also witnesses the $\lambda$-supercompactness of $\kappa$, as in statement (ii) of Lemma \ref{lemma:SupercompactLevelSmallChar}. \qed
   \end{enumerate}
\end{corollary}



\section{Internally Supercompact Cardinals}\label{section:internal}

In the following sections of this paper, we will formulate set-theoretic principles, called \emph{internal large cardinals}, that capture properties of large cardinals that accessible cardinals can possess and that are motivated by the above small embedding characterizations. In addition, we will use the theory developed in the earlier sections to derive the consistency of these principles. These principles are motivated by the observation that, in many cases, the small embeddings characterizing certain large cardinals may be lifted to suitable forcing extensions, and that those lifted embeddings retain most of the combinatorial properties of the original small embeddings. Our principles then describe the properties of these lifted embeddings to capture a strong fragment of what is left of certain large cardinals after their inaccessibility has been destroyed.  The name \emph{internal large cardinals} was chosen since the domain models of the  embeddings constructed in the consistency proofs of these principles will usually be elements of some intermediate forcing extension, i.e.\ they are \emph{internal} to the actual forcing extension. 
 In the remainder of this paper, we will study certain internal versions of ineffable, subtle and supercompact cardinals and their relationship to certain generalized tree properties. An extensive study of internal versions of many large cardinal properties and their consequences will be contained in the upcoming \cite{hl}.

In this paper, we will limit ourselves to the case of internal large cardinals with respect to forcing notions that satisfy the $\omega_1$-approximation property. For this reason, we recall the definition of this property.

\begin{definition}\label{definition:appcover}
  Given transitive classes $M\subseteq N$, the pair $(M,N)$ satisfies the \emph{$\omega_1$-approxi\-ma\-tion property} if $A\in M$ whenever $A\in N$ is such that $A\subseteq B$ for some $B\in M$, and $A\cap x\in M$ for every $x\in M$ which is countable in $M$.\footnote{In case $M$ and $N$ have the same ordinals and satisfy enough set theory, this definition is equivalent to the more common definition of the $\omega_1$-approximation property where rather than requiring $A\subseteq B$ for some $B\in M$, one only requires that $A\subseteq M$.
} 
\end{definition}

The results of this section are supposed to demonstrate that this approach leads to fruitful concepts by showing that the internal large cardinal notion corresponding to Magidor's small embedding characterization of supercompactness in Corollary \ref{corollary:SupercompactMagidorStyle} and the $\omega_1$-approximation property is equivalent to a well-studied generalized tree property.

 \begin{definition}
  A cardinal $\kappa$ is \emph{internally AP supercompact} if for all sufficiently large regular cardinals $\theta$ and all $x\in\HH{\theta}$, there is a small embedding $\map{j}{M}{\HH{\theta}}$ for $\kappa$ and a transitive model $N$ of $\ZFC^-$ such that $x\in\ran j$ and the following statements hold: 
  \begin{enumerate}[leftmargin=0.7cm]
   \item $N\subseteq\HH{\theta}$ and the pair $(N,\HH{\theta})$ satisfies the $\omega_1$-approximation property. 
   
   \item $M\in N$ and $M=\HH{\delta}^N$ for some $N$-cardinal $\delta<\kappa$. 
  \end{enumerate}
\end{definition}

%
%
%

In the remainder of this section, we will use results of Matteo Viale and Christoph Wei\ss\ from \cite{MR2838054} to show that internal AP supercompactness is equivalent to a generalized tree property. Since the results of \cite{MR2838054} show that the \emph{Proper Forcing Axiom $\PFA$} implies this tree property for $\omega_2$, these arguments will also yield a consistency proof for the internal AP supercompactness of $\omega_2$. In order to formulate these results, we need to recall the definitions of \emph{slender lists} and \emph{ineffable tree properties}.  The concept of \emph{slenderness} originates from work of Saharon Shelah, and was isolated and studied by Christoph Wei\ss\ in \cite{weissthesis}. 

\begin{definition}
 Let $\kappa$ be an uncountable regular cardinal and let $\lambda\geq\kappa$ be a cardinal. 
 \begin{enumerate}[leftmargin=0.7cm]
  \item A $\Poti{\lambda}{\kappa}$-list $\seq{d_a}{a\in\Poti{\lambda}{\kappa}}$ is \emph{slender} if for every sufficiently large cardinal $\theta$, there is a club $C$ in $\Poti{\HH{\theta}}{\kappa}$ with $b\cap d_{X\cap\lambda}\in X$ for all $X\in C$ and all $b\in X\cap\Poti{\lambda}{\omega_1}$.  
  
  \item $\ISP(\kappa,\lambda)$ is the statement that for every slender $\Poti{\lambda}{\kappa}$-list $\seq{d_a}{a\in\Poti{\lambda}{\kappa}}$, there exists $D\subseteq\lambda$ such that the set $\Set{a\in\Poti{\lambda}{\kappa}}{d_a=D\cap a}$ is stationary in $\Poti{\lambda}{\kappa}$.
 \end{enumerate}
\end{definition}

The following lemma is the main result of this section.

\begin{lemma}\label{lemma:intscisp}
  A regular cardinal $\kappa>\omega_1$ is internally approximating supercompact if and only if $\ISP(\kappa,\lambda)$ holds for all cardinals $\lambda\geq\kappa$.
\end{lemma}

In order to prove this statement, we need to introduce more concepts from \cite{MR2838054}.

\begin{definition}
  Let $\vartheta$ be an uncountable cardinal and let $X\prec\HH{\vartheta}$. 
  \begin{enumerate}[leftmargin=0.7cm]
   \item A set $d$ is \emph{$X$-appro\-xi\-ma\-ted} if $b\cap d\in X$ for all $b\in X\cap\Poti{X}{\omega_1}$. 
   
   \item A set $d$ is \emph{$X$-guessed} if $d\cap X=e\cap X$ for some $e\in X$. 
   
   \item Given $\rho\in\On$, $X$ is a \emph{$\rho$-guessing model} if every $X$-approximated $d\subseteq\rho$ is $X$-guessed. 
   
   \item $X$ is  a \emph{guessing model} if it is $\rho$-guessing for every $\rho\in X\cap\On$. 
  \end{enumerate} 
\end{definition}

\begin{proposition}\label{proposition:collap}
  Let $\vartheta>\omega_1$ be a cardinal, let $X\prec \HH{\vartheta}$ 
and let $\map{\pi}{X}{M}$ be the corresponding transitive collapse. Then $X$ is a guessing model if and only if the pair $(M,\HH{\vartheta})$ satisfies the $\omega_1$-approximation property. 
\end{proposition}

\begin{proof}
  First, assume that $X$ is a guessing model. Pick $B\in M$ and $A\in\HH{\vartheta}$ with the property that $A\subseteq B$ and $A\cap x\in M$ for every $x\in M$ that is countable in $M$. Pick a bijection $\map{f}{B}{\rho}$ in $M$ with $\rho\in\On$. We define $$d ~ = ~ (\pi^{{-}1}\circ f)[A] ~ \subseteq ~ \pi^{{-}1}(\rho) ~ \in ~  X\cap\On.$$ Fix $b\in X\cap\Poti{X}{\omega_1}$. Then $f^{{-}1}[\pi(b)\cap\rho]\in M$ is countable in $M$ and this implies that $A\cap f^{{-}1}[\pi(b)\cap\rho]\in M$.  Since elementarity implies that $\pi(b)=\pi[b]$, this yields $$b\cap d ~ =  ~ (\pi^{{-}1}\circ f)[A\cap f^{{-}1}[\pi(b)\cap\rho]]~ \in ~ X.$$ These computations show that $d\subseteq\pi^{{-}1}(\rho)$ is $X$-approximated. 
  Since $X$ is a guessing model, $d$ is $X$-guessed and there is an $e\in X$ with $d=d\cap X=e\cap X$. But then $\pi(e)=\pi[d]=f[A]\in M$ and hence $A$ is an element of $M$. 
  
  For the other direction, assume that the pair $(M,\HH{\vartheta})$ satisfies the $\omega_1$-approxi\-ma\-tion property. Pick $\rho\in X\cap\On$ and  $d\subseteq\rho$ that is $X$-approximated. Set $A=\pi[d\cap X]\subseteq\pi(\rho)\in M$. Fix an $x\in M$ that is countable in $M$. Then $\pi^{{-}1}(x)\in X\cap\Poti{X}{\omega_1}$,   $d\cap\pi^{{-}1}(x)\in X$ and therefore $$A\cap x ~ = ~ \pi[d\cap X]\cap\pi[\pi^{{-}1}(x)] ~ = ~ \pi[d\cap\pi^{{-}1}(x)] ~ = ~ \pi(d\cap\pi^{{-}1}(x))\in M,$$ because $\pi^{{-}1}(x)=\pi^{{-}1}[x]$ and $\pi(d\cap\pi^{{-}1}(x))=\pi[d\cap\pi^{{-}1}(x)]$. By our assumption, it follows that $A\in M$. Since $\pi^{{-}1}(A)\cap X=d\cap X$, we can conclude that $d$ is $X$-guessed.   
\end{proof}

We are now ready to show that the internal AP supercompactness of a cardinal $\kappa$ is equivalent to the statement that $\ISP(\kappa,\lambda)$ holds for all cardinals $\lambda\geq\kappa$.

\begin{proof}[Proof of Lemma \ref{lemma:intscisp}]
  By {\cite[Proposition 3.2 \& 3.3]{MR2838054}}, the following statements are equivalent for every uncountable regular cardinal $\kappa$ and all cardinals $\lambda\geq\kappa$:
\begin{enumerate}[leftmargin=0.7cm]
  \item $\ISP(\kappa,\lambda)$. 
  
  \item If $\vartheta$ is a cardinal with $\betrag{\HH{\vartheta}}=\lambda$, then the set of all guessing models $X\prec\HH{\vartheta}$ with $\betrag{X}<\kappa$ and $X\cap\kappa\in\kappa$ is stationary in $\Poti{\HH{\vartheta}}{\kappa}$. 
  
  \item For sufficiently large cardinals $\vartheta$, there exists a $\lambda$-guessing model $X\prec\HH{\vartheta}$ with $\betrag{X}<\kappa$, $X\cap\kappa\in\kappa$ and $\lambda^+\in X$.
\end{enumerate}

First, assume that $\kappa>\omega_1$ is a regular cardinal with the property that $\ISP(\kappa,\lambda)$ holds for all cardinals $\lambda\geq\kappa$. Fix some regular  cardinal $\theta>\kappa$ and $x\in\HH{\theta}$. Pick some $\vartheta>\betrag{\HH{\theta}}$ and use (ii) to find a guessing model $X\prec\HH{\vartheta}$ of cardinality less than $\kappa$ with $\kappa,\theta,x\in X$ and $X\cap\kappa\in\kappa$. Let $\map{\pi}{X}{N}$ denote the corresponding transitive collapse. Define  $M=\HH{\pi(\theta)}^M$ and $\map{j=\pi^{{-}1}\restriction M}{M}{\HH{\theta}}$. Then $j$ is a small embedding for $\kappa$ with $x\in\ran{j}$ and $N$ is a transitive model of $\ZFC^-$ with $N\subseteq\HH{\theta}$ and $M=\HH{\delta}^N$ for some $N$-cardinal $\delta$. Since Proposition \ref{proposition:collap} shows that the pair $(N,\HH{\theta})$ has the $\omega_1$-approximation property, we can conclude that $\kappa$ is internally AP supercompact.

In the other direction, assume that $\kappa$ is internally AP supercompact. Fix cardinals $\lambda\geq\kappa$ and $\theta>\lambda^+$. Let $\vartheta>\theta$ be a sufficiently large strong limit cardinal such that there is a small embedding $\map{j}{M}{\HH{\vartheta^+}}$ for $\kappa$ and a transitive $\ZFC^-$-model $N$ witnessing the internal AP supercompactness of $\kappa$ with respect to the pair $\langle\lambda^+,\theta\rangle$. Then there is an $N$-cardinal $\varepsilon<\kappa$ with $M=\HH{\varepsilon}^N$. Pick $\delta\in M$ with $j(\delta)=\theta$. In this situation,  elementarity implies that $\delta<\varepsilon$, $\betrag{\HH{\delta}^M}^N<\varepsilon<\kappa$ and $\HH{\delta}^M=\HH{\delta}^N$. Since the pair $(N,\HH{\vartheta})$ satisfies the $\omega_1$-approximation property, this implies that the pair $(\HH{\delta}^M,\HH{\theta})$ satisfies the $\omega_1$-approximation property. If we define $X=j[\HH{\delta}^M]$, then $X\prec\HH{\theta}$, $j^{{-}1}\restriction X$ is the transitive collapse of $X$ and Proposition \ref{proposition:collap} shows that $X$ is a guessing model satisfying $\betrag{X}\leq\betrag{\HH{\delta}^M}^N<\kappa$, $X\cap\kappa=\crit{j}\in\kappa$ and $\lambda^+\in X$. By the above equivalences, this shows that  $\ISP(\kappa,\lambda)$ holds. 
\end{proof}

\begin{corollary}
 The following statements are equivalent for every inaccessible cardinal $\kappa$: 
 \begin{enumerate}
  \item $\kappa$ is supercompact. 
  
  \item $\kappa$ is internally AP supercompact. 
 \end{enumerate}
\end{corollary}

\begin{proof}
 The implication from (i) to (ii) directly follows from a combination of Corollary \ref{corollary:SupercompactMagidorStyle} and Lemma \ref{lemma:addxtorange}. In the other direction, assume that $\kappa$ is internally AP supercompact. Then Lemma \ref{lemma:intscisp} shows that $\ISP(\kappa,\lambda)$ holds for all $\lambda\geq\kappa$. Since $\kappa$ is inaccessible, every $\Poti{\lambda}{\kappa}$-list is slender (see {\cite[Proposition 2.2]{MR2959668}}) and therefor $\kappa$ is $\lambda$-ineffable for all $\lambda\geq\kappa$. By the results of \cite{MR0295904}, this implies that $\kappa$ is supercompact.  
\end{proof}

\begin{corollary}\label{corollary:PFAinternalSupercompact}
 $\PFA$ implies that $\omega_2$ is internally AP supercompact.
\end{corollary}

\begin{proof}
  By {\cite[Theorem 4.8]{MR2838054}}, $\PFA$ implies that $\ISP(\omega_2,\lambda)$ holds for all cardinals $\lambda\ge\omega_2$. In combination with Lemma \ref{lemma:intscisp}, this yields the statement of the corollary. 
\end{proof}

The results contained in {\cite[Section 5]{MR2959668}} and Section \ref{section:application} of this paper provide alternative consistency proofs for the internal supercompactness of $\omega_2$. In addition, it is also possible to establish this property by using Itay Neeman's \emph{pure side condition forcing} (see \cite[Definition 2.4]{MR3201836}) to turn a supercompact cardinal into $\omega_2$. This argument will be presented in detail in the upcoming \cite{hln}.  


\section{Internally Subtle and Ineffable Cardinals}\label{section:intsubtleineffable}

This section contains the formulation of internal large cardinal concepts for subtlety and $\lambda$-ineffability. The following definition provides such a principle based on the small embedding characterization of subtlety in Lemma \ref{lemma:SubtleSmallChar} and the $\omega_1$-approximation property.

\begin{definition}\label{internallysubtle}
  A cardinal $\kappa$ is \emph{internally AP subtle} if for all sufficiently large regular cardinals $\theta$, all $x\in\HH{\theta}$, every club $C$ in $\kappa$, and every $\kappa$-list $\vec{d}=\seq{d_\alpha}{\alpha<\kappa}$, there is a small embedding $\map{j}{M}{\HH{\theta}}$ for $\kappa$ and a transitive model $N$ of $\ZFC^-$ such that $x,\vec{d},C\in\ran j$ and the following statements hold: 
  \begin{enumerate}[leftmargin=0.7cm]
   \item $N\subseteq\HH{\theta}$ and the pair $(N,\HH{\theta})$ satisfies the $\omega_1$-approximation property. 
   
   \item $M\in N$ and $\Poti{\crit j}{\omega_1}^N\subseteq M$. 
   
   \item If $d_{\crit{j}}\in N$, then there is an $\alpha<\crit{j}$ with $\alpha\in C$ and $d_\alpha=d_{\crit{j}}\cap\alpha$.
  \end{enumerate}
\end{definition}

In the following, we will show that the above principle implies a generalized tree property corresponding to subtlety that was formulated and studied by Christoph Wei{\ss} in his \cite{weissthesis}.

\begin{definition}
 Let $\kappa$ be an uncountable regular cardinal. 
  \begin{enumerate}[leftmargin=0.7cm]
   \item  A $\kappa$-list $\seq{d_\alpha}{\alpha<\kappa}$ is \emph{slender} if there is a club $C$ in $\kappa$ with the property that for every $\gamma\in C$ and every $\alpha<\gamma$, there is a $\beta<\gamma$ with $d_\gamma\cap\alpha=d_\beta\cap\alpha$. 
   
   \item $\SSP(\kappa)$ is the statement that for every slender $\kappa$-list $\seq{d_\alpha}{\alpha<\kappa}$ and every club $C$ in $\kappa$, there are $\alpha,\beta\in C$ such that $\alpha<\beta$ and $d_\alpha=d_\beta\cap\alpha$.  
 \end{enumerate}  
\end{definition}

\begin{lemma}\label{lemma:InternalAPSubtleImpliesSSP}
 If $\kappa$ is an internally AP subtle cardinal, then $\SSP(\kappa)$ holds. 
\end{lemma}
 
 \begin{proof}
  Fix a slender $\kappa$-list $\vec{d}=\seq{d_\alpha}{\alpha<\kappa}$   and a club $C_0$ in $\kappa$. Let $C\subseteq C_0$ be a club witnessing the slenderness of $\vec{d}$ and let $\theta$ be a sufficiently large regular cardinal such that there is a small embedding $\map{j}{M}{\HH{\theta}}$ and a transitive $\ZFC^-$-model $N$ witnessing the internal AP subtlety with respect to $\vec{d}$ and $C$. Then elementarity implies that $\crit{j}\in C\subseteq C_0$.  
  
  Assume for a contradiction that $d_\crit{j}\notin N$. Then the $\omega_1$-approximation property yields an $x\in\Poti{\crit{j}}{\omega_1}^N$ with $d_\crit{j}\cap x\notin N$. Then $x\in M$ and, since $\crit{j}$ is a regular cardinal in $M$, there is an $\alpha<\crit{j}\in C$ with $x\subseteq \alpha$. In this situation, the slenderness of $\vec{d}$ yields a $\beta<\crit{j}$ with $d_{\crit{j}}\cap\alpha=d_\beta\cap\alpha$. But then we have $$d_\crit{j} \cap x ~ = Ê~ d_\crit{j} \cap x\cap\alpha ~ = ~ d_\beta\cap x\cap\alpha.$$ Since $\vec{d}\in\ran{j}$, we have $d_\beta\in M\subseteq N$ and hence $d_\crit{j} \cap x\in N$, a contradiction. 
  
  The above computations show that $d_\crit{j}\in N$ and therefore our assumptions yield an $\alpha<\crit{j}$ with $\alpha\in C\subseteq C_0$ and $d_\alpha=d_{\crit{j}}\cap\alpha$. 
 \end{proof}

\begin{corollary}
 If $\kappa$ is an internally AP subtle cardinal, then $\kappa$ is a subtle cardinal in $\LL$. 
\end{corollary}

\begin{proof}
 This statement follows directly from {\cite[Theorem 2.4.1]{weissthesis}} and the above lemma. 
\end{proof}

The results of the next section will show that a subtle cardinal is also the upper bound for the consistency strength of the internal AP subtlety of $\omega_2$.

Next, we consider internal versions of $\lambda$-ineffability provided by  Lemma \ref{lemma:LambdaIneffableSmallChar}, and their corresponding generalized tree properties.

\begin{definition}\label{internallylambdaineffable} 
   Given cardinals $\lambda\ge\kappa$, the cardinal $\kappa$ is \emph{internally AP $\lambda$-ineffable} if for all sufficiently large regular cardinals $\theta$, all $x\in\HH{\theta}$, and every $\Poti{\lambda}{\kappa}$-list $\vec{d}=\seq{d_a}{a\in\Poti{\lambda}{\kappa}}$, there is a small embedding $\map{j}{M}{\HH{\theta}}$ for $\kappa$, an ordinal $\delta\in M$ and a transitive model $N$ of $\ZFC^-$ such that $j(\delta)=\lambda$, $x,\vec{d}\in\ran j$, $\delta<\kappa$ and the following statements hold: 
 \begin{enumerate}[leftmargin=0.7cm]
   \item $N\subseteq\HH{\theta}$ and the pair $(N,\HH{\theta})$ satisfies the $\omega_1$-approximation property. 
   
   \item $M\in N$ and $\Poti{\delta}{\omega_1}^N\subseteq M$. 
   
   \item If $j^{{-}1}[d_{j[\delta]}]\in N$, then $j^{{-}1}[d_{j[\delta]}]\in M$. 
  \end{enumerate} 
\end{definition}

\begin{lemma}\label{lemma:InternalLambdaIneffableImpliesISP}
 If $\kappa$ is an internally AP $\lambda$-ineffable cardinal, then $\ISP(\kappa,\lambda)$ holds. 
\end{lemma}

\begin{proof}
   Fix a slender $\Poti{\lambda}{\kappa}$-list $\vec{d}=\seq{d_a}{a\in\Poti{\lambda}{\kappa}}$ and a sufficiently large cardinal $\theta$ such that there is a function $\map{f}{\Poti{\HH{\theta}}{\omega}}{\Poti{\HH{\theta}}{\kappa}}$ with the property that $\clo{f}$ is a club in $\Poti{\HH{\theta}}{\kappa}$ witnessing the slenderness of $\vec{d}$. Let $\vartheta$ be a sufficiently large regular cardinal such that there is a small embedding $\map{j}{M}{\HH{\vartheta}}$ with $f\in\ran{j}$, $\delta\in M$ and a transitive $\ZFC^-$-model $N$ witnessing the internal AP $\lambda$-ineffability of $\kappa$ with respect to $\vec{d}$. Pick $\varepsilon\in M$ with $j(\varepsilon)=\theta$. As in the proof of Lemma \ref{lemma:LambdaIneffableSmallChar}, we then have $X=j[\HH{\varepsilon}^M]\in\clo{f}$. 
 
 Assume for a contradiction that $j^{{-}1}[d_{j[\delta]}]\notin N$. Then the $\omega_1$-approximation property yields an $x\in\Poti{\delta}{\omega_1}^N$ with $x\cap j^{{-}1}[d_{j[\delta]}]\notin N$. Then our assumptions imply that $x$ is an element of $\Poti{\delta}{\omega_1}^M$. But then $j(x)\in X\cap\Poti{\lambda}{\omega_1}$ and the slenderness of $\vec{d}$ implies that $j(x)\cap d_{j[\delta]}\in X$. Then we can conclude that $$x\cap j^{{-}1}[d_{j[\delta]}] ~ = ~ j^{{-1}}[j(x)\cap d_{j[\delta]}] ~ = ~ j^{{-}1}(j(x)\cap d_{j[\delta]}) ~ \in ~  M ~ \subseteq ~ N,$$ a contradiction. 
 
 The above computations show that $j^{{-}1}[d_{j[\delta]}]\in N$ and hence our assumptions imply that this set is also an element of $M$. Let $D=j(j^{{-}1}[d_{j[\delta]}])$ and assume for a contradiction that the set $S=\Set{a\in\Poti{\lambda}{\kappa}}{d_a=D\cap a}$ is not stationary in $\Poti{\lambda}{\kappa}$. By elementarity, there is a function $\map{f_0}{\Poti{\delta}{\omega}}{\Poti{\delta}{\crit{j}}}$ in $M$ such that  $\clo{j(f_0)}\cap S=\emptyset$. But then $j[\delta]\in\clo{j(f_0)}\cap S$, a contradiction.  
\end{proof}

\begin{corollary}
  If $\kappa$ is an internally AP $\kappa$-ineffable cardinal, then $\kappa$ is an ineffable cardinal in $\LL$. 
\end{corollary}

\begin{proof}
 This statement follows directly from {\cite[Theorem 2.4.3]{weissthesis}} and the above lemma. 
\end{proof}

The results of the next section will also show that the consistency of the internal AP ineffability of $\omega_2$ can be established from an ineffable cardinal.


\section{On a theorem by Christoph Wei\ss}\label{section:application}

The principles $\SSP(\kappa)$ and $\ISP(\kappa,\lambda)$ mentioned above were first formulated and studied in detail by Christoph Wei{\ss} in \cite{weissthesis} and \cite{MR2959668}. The following theorem summarizes the upper bounds for the consistency strength of these statements presented there. Remember that, given transitive classes $M\subseteq N$, the pair $(M,N)$ satisfies the \emph{$\omega_1$-covering property} if whenever $A\in N$ is countable in $N$ and $A\subseteq M$, then there is a $B\in M$ which is countable in $M$ and satisfies $A\subseteq B$.

\begin{theorem}[{\cite[Theorem 2.3.1]{weissthesis}} \& {\cite[Theorem 5.4]{MR2959668}}]\label{theorem:WeissConsResults}
  Let $\tau<\kappa\leq\lambda$ be cardinals with $\tau$ uncountable and regular, and let $\vec{\PPP}=\langle\seq{\vec{\PPP}_{{<}\alpha}}{\alpha\leq\kappa},\seq{\dot{\PPP}_\alpha}{\alpha<\kappa}\rangle$ be a forcing iteration such that the following statements hold for all inaccessible cardinals $\eta\le\kappa$:  
\begin{enumerate}
  \item $\vec{\PPP}_{{<}\eta}\subseteq\HH{\eta}$\footnote{Following \cite{MR2959668}, we make use of the convention that conditions in forcing iterations are only defined on their support.} is the direct limit of $\langle\seq{\PPP_{{<}\alpha}}{\alpha<\eta},\seq{\dot{\PPP}_\alpha}{\alpha<\eta}\rangle$ and satisfies the $\eta$-chain condition. 
  
  \item If $G$ is $\vec{\PPP}_{{<}\kappa}$-generic over $\VV$ and $G_\eta$ is the filter on $\vec{\PPP}_{{<}\eta}$ induced by $G$, then the pair $(\VV[G_\eta],\VV[G])$ satisfies the $\omega_1$-approximation property. 
  
  \item If $\alpha<\eta$, then $\PPP_{{<}\alpha}$ is definable in $\HH{\eta}$ from the parameters $\tau$ and $\alpha$. 
\end{enumerate}
 Then the following statements hold:
 \begin{enumerate}[leftmargin=0.7cm]
  \item[(1)] If $\kappa$ is a subtle cardinal, then $\mathbbm{1}_{\vec{\PPP}_{{<}\kappa}}\Vdash\SSP(\check{\kappa})$. 
  
  \item[(2)] If $\kappa$ is an ineffable cardinal, then $\mathbbm{1}_{\vec{\PPP}_{{<}\kappa}}\Vdash\ISP(\check{\kappa},\check{\kappa})$. 
  
  \item[(3)] Assume that $\vec{\PPP}$ also satisfies the following statement for all inaccessible cardinals $\eta\le\kappa$:  
   \begin{enumerate}
    \item[(iv)] If $G_\eta$ is $\vec{\PPP}_{{<}\eta}$-generic over $\VV$, then the pair $(\VV,\VV[G_\eta])$ satisfies the $\omega_1$-covering property. 
   \end{enumerate}
  Then, if $\kappa$ is $\lambda^{{<}\kappa}$-ineffable  for some cardinal $\lambda\geq\kappa$, then $\mathbbm{1}_{\vec{\PPP}_{{<}\kappa}}\Vdash\ISP(\check{\kappa},\check{\lambda})$. 
 \end{enumerate}
\end{theorem}

As pointed out in {\cite[Section 5]{MR2959668}}, William Mitchell's classical proof of the consistency of the tree property at successors of regular cardinals in \cite{MR0313057} shows that for every uncountable regular cardinal $\tau$ and every inaccessible cardinal $\kappa>\tau$, there is a forcing iteration $\vec{\PPP}$ satisfying the statements (i)-(iv) listed in Theorem \ref{theorem:WeissConsResults} such that $\mathbbm{1}_{\vec{\PPP}_{{<}\kappa}}\Vdash\anf{\check{\kappa}=\check{\tau}^+}$ and forcing with $\vec{\PPP}_{{<}\kappa}$ preserves all cardinals less than or equal to $\tau$.

In the following, we discuss what appears to be a serious problem in the arguments used to derive the above statements in \cite{weissthesis} and \cite{MR2959668}. Afterwards, we present new proofs for the statements listed in Theorem \ref{theorem:WeissConsResults} in the next section. These arguments use the small embedding characterizations of subtlety and of $\lambda$-ineffability from Section \ref{section:subtleineffable} to derive the consistency of the internal large cardinal principles from the corresponding large cardinal assumption.

We would first like to point out where the problematic step in Wei\ss's proof of statements (2) and (3) seems to be, and argue that it is indeed a problem, for Wei\ss's proof would in fact show a stronger result, one that is provably wrong. Let $\kappa$ be a $\lambda$-ineffable cardinal with $\lambda=\lambda^{{<}\kappa}$, let $\vec{\PPP}=\langle\seq{\vec{\PPP}_{{<}\alpha}}{\alpha\leq\kappa},\seq{\dot{\PPP}_\alpha}{\alpha<\kappa}\rangle$ be a forcing iteration satisfying the statements (i)-(iv) listed in Theorem \ref{theorem:WeissConsResults}, let $G$ be $\vec{\PPP}_{{<}\kappa}$-generic over $\VV$ and let $\vec{d}=\seq{d_a}{a\in\Poti{\lambda}{\kappa}^{\VV[G]}}$ be a slender $\Poti{\lambda}{\kappa}$-list in $\VV[G]$. The proofs of {\cite[Theorem 2.3.1]{weissthesis}} and {\cite[Theorem 5.4]{MR2959668}} then claim that there is a stationary subset $T$ of $\Poti{\lambda}{\kappa}$ in $\VV$ and $d\in\POT{\lambda}^{\VV[G]}$ such that $d_a=d\cap a$ holds for all $a\in T$. Since $\vec{\PPP}_{{<}\kappa}$ satisfies the $\kappa$-chain condition in $\VV$ and therefore preserves the stationarity of $T$, this argument would actually yield a strengthening of $\ISP(\kappa,\lambda)$ stating that every instance of the principle is witnessed by a stationary subset of $\Poti{\lambda}{\kappa}$ contained in the ground model $\VV$. In particular, this conclusion would imply that if $G$ is $\vec{\PPP}_{{<}\kappa}$-generic over $\VV$ and $\seq{d_\alpha}{\alpha<\kappa}$ is a $\kappa$-list in $\VV[G]$, then there is a stationary subset $S$ of $\kappa$ in $\VV$ such that $d_\alpha=d_\beta\cap\alpha$ holds for all $\alpha,\beta\in S$ with $\alpha<\beta$. The following observation shows that this statement provably fails if forcing with $\vec{\PPP}_{{<}\kappa}$ destroys the ineffability of $\kappa$.

\begin{proposition}
 Let $\langle\seq{\vec{\PPP}_{{<}\alpha}}{\alpha\leq\kappa},\seq{\dot{\PPP}_\alpha}{\alpha<\kappa}\rangle$ be a forcing iteration with the property that $\kappa$ is an uncountable regular cardinal,  $\vec{\PPP}_{{<}\kappa}$ is a direct limit and $\vec{\PPP}_{{<}\kappa}$ satisfies the $\kappa$-chain condition. Let $G$ be $\vec{\PPP}_{{<}\kappa}$-generic over $\VV$ and, given $\alpha<\kappa$, let $G_\alpha$ denote the filter on $\vec{\PPP}_{{<}\alpha}$ induced by $G$. Then one of the following statements holds: 
 \begin{enumerate}[leftmargin=0.7cm]
  \item There is an $\alpha<\kappa$ such that for all $\alpha\leq\beta<\kappa$, the partial order $\dot{\PPP}_\alpha^{G_\alpha}$ is trivial. 
  
  \item There is a slender $\kappa$-list $\seq{d_\alpha}{\alpha<\kappa}$ in $\VV[G]$ with the property that for every stationary subset $S$ of $\kappa$ in $\VV$, there are $\alpha,\beta\in S$ with $\alpha<\beta$ and $d_\alpha\neq d_\beta\cap\alpha$. 
 \end{enumerate} 
\end{proposition}

\begin{proof}
 Pick a sequence $\seq{\langle\dot{q}^0_\alpha,\dot{q}^1_\alpha\rangle}{\alpha<\kappa}$ in $\VV$ such that the following statements hold for all $\alpha<\kappa$: 
 \begin{enumerate}[leftmargin=0.7cm]
  \item $\dot{q}_{\alpha,0}$ and $\dot{q}_{\alpha,1}$ are both $\vec{\PPP}_{{<}\alpha}$-names for a condition in $\dot{\PPP}_{\alpha}$. 
  
  \item If $H$ is $\vec{\PPP}_{{<}\alpha}$-generic over $\VV$, then the conditions $\dot{q}_{\alpha,0}^H$ and $\dot{q}_{\alpha,1}^H$ are compatible in $\dot{\PPP}_\alpha^H$ if and only if the partial order $\dot{\PPP}_\alpha^H$ is trivial. 
 \end{enumerate}

 Now, assume that (i) fails, and work in $\VV[G]$. Let $\map{g}{\kappa}{\kappa}$ denote the unique function with the property that for all $\beta<\kappa$, $g(\beta)$ is the minimal ordinal greater than or equal to $\sup_{\alpha<\beta}g(\alpha)$ such that $\dot{\PPP}_{g(\beta)}^{G_{g(\beta)}}$ is a non-trivial partial order. Since $\vec{\PPP}_{{<}\kappa}$ satisfies the $\kappa$-chain condition, there is a club subset $C$ of $\kappa$ in $\VV$ with $g(\alpha)<\beta$ for all $\alpha<\beta$ whenever $\beta\in C$.  Let $\vec{d}=\seq{d_\alpha}{\alpha<\kappa}$ denote the unique $\kappa$-list with the property that $$d_\alpha=0 ~ \Longleftrightarrow ~ d_\alpha\neq 1 ~ \Longleftrightarrow ~ \dot{q}_{g(\alpha),0}^{G_{g(\alpha)}}\in G^{g(\alpha)}$$ holds for every $\alpha<\kappa$, where $G^\beta$ denotes the filter on $\dot{\PPP}_\beta^{G_\beta}$ induced by $G$ for all $\beta<\kappa$. Then $\vec{d}$ is a slender $\kappa$-list. 
 
  Assume for a contradiction that there is a stationary subset $S$ of $\kappa$ in $\VV$ such that $d_\alpha=d_\beta\cap\alpha$ holds for all $\alpha,\beta\in S$ with $\alpha<\beta$. Then there is an $i<2$ with $d_\alpha=i$ for all $\alpha\in S$. Let $\dot{g}$ be a $\vec{\PPP}_{{<}\kappa}$-name for a function from $\kappa$ to $\kappa$ with $g=\dot{g}^G$ and let $\dot{d}$ be a $\vec{\PPP}_{{<}\kappa}$-name for a $\kappa$-list with $\vec{d}=\dot{d}^G$. Let $p$ be a condition in $G$ forcing all of the above statements. Pick a condition $q$ in $\vec{\PPP}_{{<}\kappa}$ below $p$
. Then there is $\alpha\in C\cap S$ with $q\in\vec{\PPP}_{{<}\alpha}$. 
By density, we can find a condition $s\in G$ below $q$, and $\alpha\leq\beta<\kappa$ with $g(\alpha)=\beta$, $s\in\vec{\PPP}_{{<}\beta+1}$ and $s(\beta)=\dot{q}_{\beta,1-i}$. But then $\dot{\PPP}_\beta^{G_\beta}$ is non-trivial, $\dot{q}_{\beta,1-i}^{G_\beta}\in G^\beta$ and $d_\alpha=1-i$, a contradiction. 
\end{proof}

In the argument that is supposed to prove the above statement, Wei{\ss} constructs a club $C$ in $\Poti{\lambda}{\kappa}$ in $\VV$ such that $d_a\in\VV[G_{a\cap\kappa}]$ holds for every $a\in C$ with the property that $a\cap\kappa$ is an inaccessible cardinal in $\VV$.  The problematic step then seems to be his conclusion that there exists a sequence $\seq{\dot{d}_a}{a\in C}$ in $\VV$ with the property that for all  $a\in C$ with $a\cap\kappa$ inaccessible in $\VV$, $\dot{d}_a$ is a $\vec{\PPP}_{{<}(a\cap\kappa)}$-name with $d_a=\dot{d}_a^G$.  This conclusion is problematic, because assuming the existence of such a sequence of names in $\VV$, it is easy to code the name $\dot{d}_a$ as a subset of $a$ and then use the $\lambda$-ineffability of $\kappa$ in $\VV$ to obtain a stationary subset of $\Poti{\lambda}{\kappa}$ in $\VV$ that witnesses the strengthening of $\ISP(\kappa,\lambda)$ formulated above. Therefore the above observation shows that such a sequence cannot exist in the ground model $\VV$. Since a similar argument is used in the proof of statement (1) of Theorem \ref{theorem:WeissConsResults} in \cite{weissthesis}, it is also not clear if these arguments can be modified to produce a correct proof of the statement.

In the next section, we will use the theory of small embeddings developed in this paper to present a different proof of the three statements listed in Theorem \ref{theorem:WeissConsResults}.

\section{The consistency of internal subtlety and of internal $\lambda$-ineffability}\label{section:ConsInternal}

Based on our small embedding characterizations of subtlety and of $\lambda$-ineffability, we will provide consistency proofs of the internal large cardinal principles introduced in Section \ref{section:intsubtleineffable}. By the results of that section, these proofs will in particular yield a slight strengthening of the statements listed in Theorem \ref{theorem:WeissConsResults}. In contrast to these statements, the results of this section do not rely on any kind of definability assumption and, in the case of $\lambda$-ineffable cardinals, we will not need to assume any kind of covering property of our iteration.

In combination with Lemma \ref{lemma:InternalAPSubtleImpliesSSP}, the following theorem directly yields a proof of statement (1) of Theorem \ref{theorem:WeissConsResults}. As already mentioned above, the results of \cite{MR0313057} show that there are forcing iterations with these properties that turn inaccessible cardinals into the successor of an uncountable regular cardinal. In particular, it is possible to establish the consistency of $\SSP(\omega_2)$ from a subtle cardinal.

\begin{theorem}\label{theorem:ForceInternalAPLambdaIneffable}
  Let $\vec{\PPP}=\langle \seq{\vec{\PPP}_{{<}\alpha}}{\alpha\leq\kappa},\seq{\dot{\PPP}_\alpha}{\alpha<\kappa}\rangle$ be a forcing iteration with $\kappa$ an uncountable and regular cardinal, such that the following statements hold for all inaccessible $\nu\leq\kappa$: 
  
  \begin{enumerate}[leftmargin=0.7cm] 
    \item $\vec{\PPP}_{{<}\nu}\subseteq\HH{\nu}$ is the direct limit of $\langle \seq{\vec{\PPP}_{{<}\alpha}}{\alpha<\nu},\seq{\dot{\PPP}_\alpha}{\alpha<\nu}\rangle$ and satisfies the $\nu$-chain condition. 
    
    \item If $G$ is $\vec{\PPP}_{{<}\kappa}$-generic over $\VV$ and $G_\nu$ is the filter on $\vec{\PPP}_{{<}\nu}$ induced by $G$, then the pair $(\VV[G_\nu],\VV[G])$ satisfies the $\omega_1$-approximation property.  
  \end{enumerate} 
If $\kappa$ is a subtle cardinal, then $\mathbbm{1}_{\vec{\PPP}_{{<}\kappa}}\Vdash\anf{\textit{$\kappa$ is internally AP subtle}}$.
\end{theorem}

\begin{proof}
 Let $\dot{d}$ be a $\vec{\PPP}_{{<}\kappa}$-name for a $\kappa$-list, let $\dot{C}$ be a $\vec{\PPP}_{{<}\kappa}$-name for a club in $\kappa$ and let $\dot{x}$ be any $\vec{\PPP}_{{<}\kappa}$-name. Since $\vec{\PPP}_{{<}\kappa}$ satisfies the $\kappa$-chain condition, there is a club $C\subseteq\Lim$ in $\kappa$ such that  $\mathbbm{1}_{\vec{\PPP}_{{<}\kappa}}\Vdash\anf{\check{C}\subseteq\dot{C}}$ and all elements of $C$ are closed under the G{\"o}del pairing function $\goedel{\cdot}{\cdot}$. Given $\alpha<\kappa$, let $\dot{d}_\alpha$ be a $\vec{\PPP}_{{<}\kappa}$-nice name for the $\alpha$-th component of $\dot{d}$. Pick a regular cardinal $\theta>2^\kappa$ with $\dot{d},\dot{x},C,\dot{C},\vec{\PPP}\in\HH{\theta}$, which is sufficiently large with respect to Statement (ii) in Lemma \ref{lemma:SubtleSmallChar}.  
 Let $G$ be $\vec{\PPP}_{{<}\kappa}$-generic over $\VV$. 
 
 First, assume that there is an inaccessible cardinal $\nu<\kappa$ in $\VV$ and a small embedding $\map{j}{M}{\HH{\theta}^\VV}$ for $\kappa$ in $\VV$ such that $\dot{d},\dot{x},C,\dot{C},\vec{\PPP}\in\ran{j}$, $\nu=\crit{j}$ and $\dot{d}_\nu^G\notin\VV[G_\nu]$. Then our assumptions on $\vec{\PPP}$ imply that $\vec{\PPP}_{{<}\nu}\in M$, $j(\vec{\PPP}_{{<}\nu})=\vec{\PPP}_{{<}\kappa}$ and $j\restriction\vec{\PPP}_{{<}\nu}=\id_{\vec{\PPP}_{{<}\nu}}$. Hence it is possible to lift $j$ in order to obtain a small embedding $\map{j_*}{M[G_{\nu}]}{\HH{\theta}^{\VV[G]}}$ for $\kappa$ in $\VV[G]$ with $\dot{d}^G,\dot{x}^G,\dot{C}^G\in\ran{j_*}$. Set $N=\HH{\theta}^{\VV[G_{\nu}]}$. Then our assumptions imply that $M[G_\nu]\in N\subseteq\HH{\theta}^{\VV[G]}$ and the pair $(N,\HH{\theta}^{\VV[G]})$ satisfies the $\omega_1$-approximation property. Moreover, we also have $\HH{\nu}^N\subseteq M[G_\nu]$, because Lemma \ref{lemma:ModelsContainInitialSegment} shows that $\HH{\nu}^\VV\subseteq M$ and $\vec{\PPP}_{{<}\nu}$ satisfies the $\nu$-chain condition in $\VV$. Since $\dot{d}_\nu^G\notin\VV[G_\nu]$, this shows that the embedding $j_*$ and the model $N$ witness that $\kappa$ is internally $AP$ subtle with respect to $\dot{d}^G$, $\dot{x}^G$ and $\dot{C}^G$ in $\VV[G]$. 

Next, assume that $\dot{d}_\nu^G\in\VV[G_\nu]$ holds for every $\nu$ contained in the set $A$ of all inaccessible cardinals $\nu<\kappa$ in $\VV$ with the property that there is small embedding $\map{j}{M}{\HH{\theta}^\VV}$ for $\kappa$ in $\VV$ with $\crit{j}=\nu$ and $\dot{d},\dot{x},C,\dot{C},\vec{\PPP}\in\ran{j}$. Let $p\in G$ be a condition forcing this statement. Work in $\VV$ and pick a condition $q$ below $p$ in $\vec{\PPP}_{{<}\kappa}$. We let $A_*$ denote the set of all $\nu\in A$ with $q\in\vec{\PPP}_{{<}\nu}$. With the help of our assumption and the fact that $\vec{\PPP}_{{<}\kappa}$ satisfies the $\kappa$-chain condition, we find a function $\map{g}{A_*}{\kappa}$ and sequences $\seq{q_\nu}{\nu\in A_*}$, $\seq{\dot{r}_\nu}{\nu\in A_*}$ and $\seq{\dot{e}_\nu}{\nu\in A_*}$ such that the following statements hold for all $\nu\in A_*$: 
 \begin{enumerate}[leftmargin=0.7cm] 
  \item[(1)] $g(\nu)>\nu$ and $\dot{d}_\nu$ is a $\vec{\PPP}_{{<}g(\nu)}$-name. 
 
  \item[(2)] $q_\nu$ is a condition in $\vec{\PPP}_{{<}\nu}$ below $q$. 
  
  \item[(3)] $\dot{r}_\nu$ is a $\vec{\PPP}_{{<}\nu}$-name for a condition in the corresponding tail forcing $\dot{\PPP}_{[\nu,g(\nu))}$. \footnote{Let us point out that the problematic argument in Wei\ss's original proof can be seen as him assuming that the name $\dot{r}_\nu$ is just the name for the trivial condition in the corresponding tail forcing.} 
  
  \item[(4)] $\dot{e}_\nu$ is a $\vec{\PPP}_{{<}\nu}$-name for a subset of $\nu$ with $\langle q_\nu,\dot{r}_\nu\rangle\Vdash_{\vec{\PPP}_{{<}\nu}*\dot{\PPP}_{[\nu,g(\nu))}}\anf{\dot{d}_\nu=\dot{e}_\nu}$. 
 \end{enumerate}
 Given $\nu\in A_*$, let $E_\nu$ denote the set of all triples $\langle s,\beta,i\rangle\in\vec{\PPP}_{{<}\nu}\times\nu\times 2\subseteq\HH{\nu}$ with $$s\Vdash_{\vec{\PPP}_{{<}\nu}}\anf{\check{\beta}\in \dot{e}_\nu ~ \longleftrightarrow  ~ i=1}.$$
 
Let $\vec{c}=\seq{c_\alpha}{\alpha<\kappa}$ be the $\kappa$-list, and let $C_*$ be the club in $\kappa$, obtained from an application of Lemma \ref{Lemma:CritIndes}. Fix a bijection $\map{f}{\kappa}{\HH{\kappa}}$ with $f[\nu]=\HH{\nu}$ for every inaccessible cardinal $\nu<\kappa$. Let $\vec{d}=\seq{d_\alpha}{\alpha<\kappa}$ be the unique $\kappa$-list such that the following statements hold for all $\alpha<\kappa$: 
 \begin{enumerate}[leftmargin=0.7cm] 
  \item[(a)] If $\alpha\in A_*$, then $$d_\alpha ~ = ~ \{\goedel{0}{0}\}\cup\{\goedel{f^{{-}1}(q_\alpha)}{1}\}\cup\Set{\goedel{f^{{-}1}(e)}{2}}{e\in E_\alpha} ~ \subseteq ~ \alpha.$$
  
  \item[(b)] If $\omega\subseteq\alpha\notin A_*$ and $\alpha$ is closed under $\goedel{\cdot}{\cdot}$, then $$d_\alpha ~ = ~ \{\goedel{1}{0}\}\cup\Set{\goedel{\beta}{1}}{\beta\in c_\alpha} ~ \subseteq ~ \alpha.$$
  
  \item[(c)] Otherwise, $d_\alpha$ is the empty set. 
 \end{enumerate}

 Let $\map{j}{M}{\HH{\theta}}$ be a small embedding for $\kappa$ witnessing the subtlety of $\kappa$ with respect to $\vec{d}$ and $C\cap C_*$, as in Statement (ii) of Lemma \ref{lemma:SubtleSmallChar}, such that $\vec{c},\dot{d},\vec{e},f,g,q,\dot{x},C,C_*,\dot{C},\vec{\PPP}\in\ran{j}$. Set $\nu=\crit{j}$ and pick $\alpha\in C\cap C_*\cap\nu$ with $d_\alpha=d_\nu\cap\alpha$. Then $\omega\leq\alpha<\nu$ and both $\alpha$ and $\nu$ are closed under $\goedel{\cdot}{\cdot}$. 
 
 Assume for a contradiction that $\nu\notin A_*$. This implies that $\goedel{1}{0}\in d_\alpha$ and therefore $\alpha\notin A_*$. But then $c_\alpha=c_\nu\cap\alpha$, and  $j$ witnesses the subtlety of $\kappa$ with respect to $\vec{c}$ and $C_*$, as in statement (ii) of Lemma \ref{lemma:SubtleSmallChar}. By Lemma \ref{Lemma:CritIndes}, this implies that $\nu$ is inaccessible, and hence $j$ witnesses that $\nu$ is an element of $A_*$, a contradiction. 

Hence $\nu\in A_*$, and this implies that $\goedel{0}{0}\in d_\alpha$, $\alpha\in A_*$, $g(\alpha)<\nu$, $q_\alpha=q_\nu\in\vec{\PPP}_{{<}\alpha}$ and $E_\alpha\subseteq E_\nu$. Pick a condition $u$ in $\vec{\PPP}_{{<}\kappa}$ such that the canonical condition in $\vec{\PPP}_{{<}\alpha}*\dot{\PPP}_{[\alpha,\nu)}$ corresponding to $u\restriction\nu$ is stronger than $\langle q_\alpha,\dot{r}_\alpha\rangle$ and the canonical condition in $\vec{\PPP}_{{<}\nu}*\dot{\PPP}_{[\nu,\kappa)}$ corresponding to $u$ is stronger than $\langle u\restriction\nu,\dot{r}_\nu\rangle$. Let $H$ be $\vec{\PPP}_{{<}\kappa}$-generic over $\VV$ with $u\in H$, let $H_\alpha$ denote the filter on $\vec{\PPP}_{{<}\alpha}$ induced by $H$ and let $H_\nu$ denote the filter on $\vec{\PPP}_{{<}\nu}$ induced by $H$. Then $\dot{d}_\alpha^H=\dot{e}_\alpha^{H_\alpha}\in\VV[H_\alpha]$, and $\dot{d}_\nu^H=\dot{e}_\nu^{H_\alpha}\in\VV[H_\nu]$. 
If $\beta\in \dot{d}_\alpha^H$, then there is $s\in H_\alpha\subseteq H_\nu$ with $\langle s,\beta,1\rangle\in E_\alpha\subseteq E_\nu$, and this implies that $\beta\in\dot{d}_\nu^H$. In the other direction, if $\beta\in\alpha\setminus\dot{d}_\alpha^H$, then there is $s\in H_\alpha$ with $\langle s,\beta,0\rangle\in E_\alpha$, and hence $\beta\notin\dot{d}_\nu^H$. This shows that $\dot{d}_\alpha^H=\dot{d}_{\crit{j_*}}^H\cap\alpha$. Set $N=\HH{\theta}^{\VV[H_\nu]}$ and let $\map{j_*}{M[H_\nu]}{\HH{\theta}^{\VV[H]}}$ denote the lift of $j$ in $\VV[H]$. Then $j_*$ is a small embedding for $\kappa$ in $\VV[H]$ with $\dot{d}^H,\dot{x}^H,\dot{C}^H\in\ran{j}$, $M[H_\nu]\in N\subseteq \HH{\theta}^{\VV[H]}$, and the pair $(N,\HH{\theta}^{\VV[H]})$ satisfies the $\omega_1$-approximation property. Moreover, since Lemma \ref{lemma:ModelsContainInitialSegment} shows that $\HH{\nu}^\VV\subseteq M$ and $\vec{\PPP}_{{<}\nu}$ satisfies the $\nu$-chain condition in $\VV$, we also have $\HH{\nu}^N\subseteq M[H_\nu]$. Finally, we have $\alpha\in\dot{C}^H\cap\crit{j_*}$ with $\dot{d}_\alpha^H=\dot{d}_{\crit{j_*}}^H\cap\alpha$. 

 Since $u\leq_{\vec{\PPP}_{{<}\kappa}}q$ holds in the above computations, a density argument shows that there is a small embedding $\map{j}{M}{\HH{\theta}^{\VV[G]}}$ for $\kappa$ in $\VV[G]$ witnessing that $\kappa$ is internally AP subtle with respect to $\dot{d}^G$, $\dot{x}^G$ and $\dot{C}^G$ in $\VV[G]$. 
\end{proof}

A variation of the above proof, using Lemma \ref{lemma:LambdaIneffableClosureProperties}, allows us to establish the consistency of internal AP $\lambda$-ineffability for accessible cardinals. Note that since $\lambda^{{<}\kappa}=(\lambda^{{<}\kappa})^{{<}\kappa}$ and $\ISP(\kappa,\lambda^{{<}\kappa})$ implies $\ISP(\kappa,\lambda)$ (see {\cite[Proposition 3.4]{MR2959668}}), a combination of the following result and Lemma \ref{lemma:InternalLambdaIneffableImpliesISP} implies the statements (2) and (3) listed in Theorem \ref{theorem:WeissConsResults}. Moreover, note that results of Chris Johnson in \cite{MR1111312} show that if $\kappa$ is $\lambda$-ineffable and $\cof\lambda\geq\kappa$, then $\lambda=\lambda^{{<}\kappa}$  (see also \cite[Proposition 1.5.4]{weissthesis}).

\begin{theorem}\label{correctedlambdaineffable}
 Let $\kappa$ be a cardinal, and let $\vec{\PPP}=\langle \seq{\vec{\PPP}_{{<}\alpha}}{\alpha\leq\kappa},\seq{\dot{\PPP}_\alpha}{\alpha<\kappa}\rangle$ be a forcing iteration satisfying the statements listed in Theorem \ref{theorem:ForceInternalAPLambdaIneffable}. If $\kappa$ is a $\lambda$-ineffable cardinal with $\lambda=\lambda^{{<}\kappa}$, then $\mathbbm{1}_{\vec{\PPP}_{{<}\kappa}}\Vdash\anf{\textit{$\kappa$ is internally AP $\lambda$-ineffable}}$. 
\end{theorem}

\begin{proof}
 Let $\dot{d}$ be a $\vec{\PPP}_{{<}\kappa}$-name for a $\Poti{\lambda}{\kappa}$-list and let $\dot{x}$ be any $\vec{\PPP}_{{<}\kappa}$-name. Given $a\in\Poti{\lambda}{\kappa}$, let $\dot{d}_a$ be a $\vec{\PPP}_{{<}\kappa}$-nice name for the $a$-th component of $\dot{d}$. Fix a bijection $\map{f}{\kappa}{\HH{\kappa}}$ with $f[\nu]=\HH{\nu}$ for every inaccessible cardinal $\nu<\kappa$. Pick a regular cardinal $\theta>2^\lambda$ with $\dot{d},\dot{x},\vec{\PPP}\in\HH{\theta}$, which is sufficiently large with respect to Statement (ii) in Lemma \ref{lemma:LambdaIneffableSmallChar}.  
 Let $G$ be $\vec{\PPP}_{{<}\kappa}$-generic over $\VV$. 

First, assume that there is an inaccessible cardinal $\nu<\kappa$ in $\VV$, a small embedding $\map{j}{M}{\HH{\theta}^\VV}$ for $\kappa$ in $\VV$ and $\delta\in M\cap\kappa$ such that $\dot{d},f,\dot{x},\vec{\PPP}\in\ran{j}$, $\nu=\crit{j}$, $j(\delta)=\lambda$, $\Poti{\delta}{\nu}^\VV\subseteq M$ and $\dot{d}_{j[\delta]}^G\notin\VV[G_\nu]$. Let $\map{j_*}{M[G_{\nu}]}{\HH{\theta}^{\VV[G]}}$ denote the corresponding lift of $j$. Then $j_*$ is a small embedding for $\kappa$ in $\VV[G]$ with $\dot{d}^G,\dot{x}^G\in\ran{j_*}$. If we define $N=\HH{\theta}^{\VV[G_{\nu}]}$, then our assumptions imply that $M[G_\nu]\in N\subseteq\HH{\theta}^{\VV[G]}$ and the pair $(N,\HH{\theta}^{\VV[G]})$ satisfies the $\omega_1$-approximation property.  Moreover, since $\Poti{\delta}{\nu}^\VV\subseteq M$ and $\vec{\PPP}_{{<}\nu}$ satisfies the $\nu$-chain condition in $\VV$, we know that $\Poti{\delta}{\nu}^N\subseteq M[G_\nu]$.  Since  $\dot{d}_{j[\delta]}^G\notin\VV[G_\nu]$ implies that $j_*^{{-}1}[\dot{d}_{j_*[\delta]}^G]\notin N$, we can conclude that $j_*$, $\delta$ and $N$ witness that $\kappa$ is internally AP $\lambda$-ineffable with respect to $\dot{d}^G$ and $\dot{x}^G$ in $\VV[G]$. 

Next, assume that $\dot{d}_a^G\in\VV[G_\nu]$ holds for all elements of the set $A$ of all $a\in\Poti{\lambda}{\kappa}^\VV$ with the property that there is small embedding $\map{j}{M}{\HH{\theta}^\VV}$ for $\kappa$ in $\VV$ and $\delta\in M\cap\kappa$ with $j(\delta)=\lambda$, $a=j[\delta]$, $\nu_a=\crit{j}=a\cap\kappa$ is an inaccessible cardinal in $\VV$, $\Poti{\delta}{\nu_a}^\VV\subseteq M$ and $\dot{d},f,\dot{x},\vec{\PPP}\in\ran{j}$. Let $p\in G$ be a condition forcing this statement. Work in $\VV$ and pick a condition $q$ below $p$ in $\vec{\PPP}_{{<}\kappa}$. We let $A_*$ denote the set of all $a\in A$ with $q\in\vec{\PPP}_{{<}\nu_a}$. Then all elements of $A_*$ are closed under $\goedel{\cdot}{\cdot}$ and, with the help of our assumption and the fact that $\vec{\PPP}_{{<}\kappa}$ satisfies the $\kappa$-chain condition, we find sequences $\seq{q_a}{a\in A_*}$, $\seq{\dot{r}_a}{a\in A_*}$ and $\seq{\dot{e}_a}{a\in A_*}$ such that the following statements hold for all $a\in A_*$: 
 \begin{enumerate}[leftmargin=0.7cm]  
  \item[(1)] $q_a$ is a condition in $\vec{\PPP}_{{<}\nu_a}$ below $q$. 
  
  \item[(2)] $\dot{r}_a$ is a $\vec{\PPP}_{{<}\nu_a}$-name for a condition in the corresponding tail forcing $\dot{\PPP}_{[\nu_a,\kappa)}$.\footnote{Let us point out that the problematic argument in Wei\ss's original proof can be seen as him assuming that the name $\dot{r}_a$ is just the name for the trivial condition in the corresponding tail forcing.} 
  
  \item[(3)] $\dot{e}_a$ is a $\vec{\PPP}_{{<}\nu_a}$-nice name for a subset of $a$ with $\langle q_a,\dot{r}_a\rangle\Vdash_{\vec{\PPP}_{{<}\nu_a}*\dot{\PPP}_{[\nu_a,\kappa)}}\anf{\dot{d}_a=\dot{e}_a}$. 
 \end{enumerate}
 
 Let $\vec{c}=\seq{c_a}{a\in\Poti{\lambda}{\kappa}}$ be the $\Poti{\lambda}{\kappa}$-list given by  Lemma \ref{lemma:LambdaIneffableClosureProperties} and let $\vec{d}=\seq{d_a}{a\in\Poti{\lambda}{\kappa}}$ be the unique $\Poti{\lambda}{\kappa}$-list with $$d_a ~ = ~ \Set{\goedel{f^{{-}1}(s)}{\beta}}{\langle\check{\beta},s\rangle\in\dot{e}_a} ~ \subseteq ~ a$$ for all $a\in A_*$ and $d_a=c_a$ for all $a\in\Poti{\lambda}{\kappa}\setminus A_*$. Pick a small embedding $\map{j}{M}{\HH{\theta}}$ for $\kappa$ and $\delta\in M\cap\kappa$ that witness the $\lambda$-ineffability of $\kappa$ with respect to $\vec{d}$, as in Statement (ii) of Lemma \ref{lemma:LambdaIneffableSmallChar}, such that $\vec{c},\dot{d},f,q,\dot{x},\vec{\PPP}\in\ran{j}$. 
 
Assume for a contradiction that $j[\delta]\notin A_*$. Then $d_{j[\delta]}=c_{j[\delta]}$ and hence $j^{{-}1}[c_{j[\delta]}]\in M$. This shows that $j$ and $\delta$ witness the $\lambda$-ineffability of $\kappa$ with respect to $\vec{c}$, and Lemma \ref{lemma:LambdaIneffableClosureProperties} implies that $\crit{j}$ is an inaccessible cardinal and $\Poti{\delta}{\crit{j}}\subseteq M$. But then $j$ and $\delta$ also witness that $j[\delta]$ is an element of $A_*$, a contradiction.  

Hence $j[\delta]\in A_*$. Pick a condition $u$ in $\vec{\PPP}_{{<}\kappa}$ such that the canonical condition in $\vec{\PPP}_{{<}\nu_{j[\delta]}}*\dot{\PPP}_{[\nu_{j[\delta]},\kappa)}$ corresponding to $u$ is stronger than $\langle q_{j[\delta]},\dot{r}_{j[\delta]}\rangle$. Let $H$ be $\vec{\PPP}_{{<}\kappa}$-generic over $\VV$ with $u\in H$ and let $H_j$ denote the filter on $\vec{\PPP}_{{<}\nu_{j[\delta]}}$ induced by $H$. Then $\dot{d}_{j[\delta]}^H=\dot{e}_{j[\delta]}^{H_j}\in\VV[H_j]$. Given $\gamma<\delta$, we have $j(\gamma)\in\dot{d}_{j[\delta]}^H$ if and only if there is an $s\in H_j$ with $\goedel{f^{{-}1}(s)}{j(\beta)}\in d_{j[\delta]}$. Since $f\restriction\nu_{j[\delta]}\in M$ with $j(f\restriction\nu_{j[\delta]})=f$, this shows that $j^{{-}1}[\dot{d}_{j[\delta]}^H]$ is equal to the set of all $\gamma<\delta$ such that there is an $s\in H_j$ with $\goedel{(f\restriction\nu_{j[\delta]})^{{-}1}(s)}{\gamma}\in j^{{-}1}[d_{j[\delta]}]$. This shows that $j^{{-}1}[\dot{d}_{j[\delta]}^H]$ is an element of $M[H_j]$. 
 
 Set $N=\HH{\theta}^{\VV[H_j]}$ and let $\map{j_*}{M[H_j]}{\HH{\theta}^{\VV[H]}}$ denote the induced lift of $j$. Then $j_*$ is a small embedding for $\kappa$ in $\VV[H]$ such that $\dot{d}^H,\dot{x}^H,\in\ran{j}$, $\dot{d}_{j_*[\delta]}^H\in M[H_j]$, $M[H_j]\in N\subseteq \HH{\theta}^{\VV[H]}$ and the pair $(N,\HH{\theta}^{\VV[H]})$ satisfies the $\omega_1$-approximation property. Since $\Poti{\delta}{\nu_{j[\delta]}}^\VV\subseteq M$ and $\vec{\PPP}_{{<}\nu_{j[\delta]}}$ satisfies the $\nu_{j[\delta]}$-chain condition in $\VV$, we also know that $\Poti{\delta}{\crit{j_*}}^N\subseteq M[H_j]$. 

 As above, a density argument shows that there is a small embedding $\map{j}{M}{\HH{\theta}^{\VV[G]}}$ for $\kappa$ in $\VV[G]$ witnessing that $\kappa$ is internally AP $\lambda$-ineffable with respect to $\dot{d}^G$ and $\dot{x}^G$ in $\VV[G]$. 
\end{proof}

\section{Open Questions}\label{section:openquestions}

Clearly, our paper suggests the following task.

\begin{question}
  Find small embedding characterizations for large cardinal notions other than those presented in this paper, for example for extender-based large cardinals like strong or Woodin cardinals!
\end{question}

In Section \ref{section:subtleineffable}, we introduced the principles of supersubtle and $\lambda$-superineffable cardinals. While standard arguments showed that supersubtle cardinals are downwards absolute to $\LL$ in Lemma \ref{supersubtleinl}, similar arguments seem not to work for $\kappa$-superineffable cardinals. So we ask the following.

\begin{question}\label{superineffablequestion}
  Are $\kappa$-superineffable cardinals downwards absolute to $\LL$?
\end{question}

 \bibliographystyle{plain}
\bibliography{references}

\end{document}